\documentclass[a4paper,11pt]{article}

\usepackage[latin1]{inputenc}
\usepackage{amsfonts,amsmath,amssymb,amsthm,graphicx,epsfig,float}
\usepackage[T1]{fontenc}
\usepackage[english,francais]{babel}
\usepackage[all]{xy}

\usepackage{amsthm}
\newtheorem{Thm}{Theorem}
\newtheorem{Pro}{Proposition}[section]
\newtheorem{Lem}{Lemma}[section]

\theoremstyle{definition}
\newtheorem{Def}{Definition}[section]
\theoremstyle{remark}
\newtheorem*{Rem}{Remark}
\newtheorem*{Proof}{Proof}

\DeclareMathOperator{\e}{e}

\newcommand{\Cc}{\mathbb{C}}
\newcommand{\Nn}{\mathbb{N}}

\newcommand{\Zz}{\mathbb{Z}}
\newcommand{\Pp}{\mathbb{P}}

\title{{\bf Dimensional preimage entropies}}
\author{Henry De Thélin}
\date{}

\begin{document}
\maketitle

\def\figurename{{Fig.}}%
\def\proofname{Preuve}
\def\contentsname{Sommaire}%

\selectlanguage{english}
\begin{center}
{\bf{ }}
\end{center}

\begin{abstract}

Let $X$ be a compact complex manifold of dimension $k$ and $f:X \longrightarrow X$ be a dominating meromorphic map. We generalize the notion of topological entropy, by defining a quantity $h_{(m,l)}^{top}(f)$ which measures the action of $f$ on local analytic sets $W$ of dimension $l$ with $W \subset f^{-n}(\Delta)$ where $\Delta$ is a local analytic set of dimension $m$.

We give then inequalities between $h_{(m,l)}^{top}(f)$ and Lyapounov exponents of suitable invariant measures.

\end{abstract}



Classification: 37B40, 37F10, 32Bxx.

\section*{{\bf Introduction}}
\par

Let $X$ be a compact complex manifold of dimension $k$ and $f:X \longrightarrow X$ be a dominating meromorphic map. We denote by $I$ its indeterminacy set.

A fundamental quantity to study this dynamical system is the topological entropy $h_{top}(f)$. Roughly speaking, it measures the number of different orbits that we can distinguish in the dynamical system.

The topological entropy has been generalized in several directions.

One of them is due to M. Hurley (\cite{Hu}), who defined a notion of pointwise preimage entropy which measure the number of orbits that we can distinguish in a preimage sets of individual points (see \cite{NiPr}, \cite{FFN} and \cite{ChNe} too):

$$h_{m}^{top}(f)=\lim_{\delta \to 0} \overline{\lim_n} \frac{1}{n} \log \sup_{x \in X } ( \max \# E \mbox{ , } E \mbox{   } (n,\delta) \mbox{-separated   } E \subset f^{-n}(x) ).$$

In \cite{DetVig}, with G. Vigny, we extend the definition of topological entropy  in another way: instead of considering the action of $f$ on points, we measure the action on local analytic sets of dimension $l$ (with $0 \leq l \leq k$). In some sense, we count the number of orbits of local analytic sets that we can distinguish in the dynamical system.

One of the goal of this paper is to unify and generalize these two notions.

The idea is to consider the action of $f$ on local analytic sets of dimension $l$ which are themself contained in preimages of local analytic sets with dimension $m$. It defines the dimensional preimage entropies $h_{(m,l)}^{top}(f)$.

More precisely, for $0 \leq l \leq k$, $\delta >0$ and $n \in \Nn$, we consider

\begin{equation*}
\begin{split}
&X_l^{\delta , n}= \{ W \subset X \mbox{ , } W \mbox{   is a graph of a holomorphic map   } \Phi \\
& \mbox{   over a ball  } B_l(x, e^{-\delta n}) \mbox{ , } Lip( \Phi) \leq 1 \mbox{   and   } \forall k \leq n-1 \mbox{   } f^k(W) \subset X \setminus I \}. \\
\end{split}
\end{equation*}

Here $B_l(x, e^{-\delta n})$ denotes the ball of dimension $l$ (i.e. in a chart, $B_l(x, e^{-\delta n})$ is contained in a complex plane of dimension $l$) with center $x$ and radius $e^{-\delta n}$. In all this paper, when we work in a chart, we consider the metric associated to the norm $\|x\|=\max(|x_1|, \cdots , |x_k|)$.  

We also define 
\begin{equation*}
\begin{split}
&X_l^{\delta}= \{ W \subset X \mbox{ , } W \mbox{   is a graph of a holomorphic map   } \Phi \\
& \mbox{   over a ball  } B_l(x, \delta) \mbox{ , } Lip( \Phi) \leq 1 \mbox{  and  } \Phi(x)=0 \}.
\\
\end{split}
\end{equation*}

In both definitions, by convention, $B_0(x, r)=\{x\}$, $X_0^{\delta , n}= \{ x \in X \mbox{   } \forall k \leq n-1 \mbox{   } f^k(x) \in X \setminus I \}$ and $X_0^{\delta}=X$. By convention too, a graph over $B_k(x,r)$ is $B_k(x,r)$ itself.

For $A,B \subset X$, we define $d(A,B)=\inf \{ d(x,y) \mbox{ , } x \in A \mbox{ , } y \in B \}$. As in \cite{DetVig}, we extend the definition of $(n, \delta)$-separated points to the analytic sets in $X_l^{\delta , n}$:

\begin{Def}
A set $E \subset X_l^{\delta , n}$ is $(n, \delta)$-separated if for all $W \neq W'$ in $E$ we have $\displaystyle \max_{i=0, \cdots , n-1} d(f^{i}(W),f^{i}(W')) \geq \delta$.
\end{Def}

We can give now the definition of the dimensional preimage entropies: the $(m,l)$-topological entropy counts the maximal number of elements $W$ in $X_l^{\delta , n}$ which are $(n, \delta)$-separated and contained in a set $f^{-n}(\Delta)$ (with $\Delta \in X_m^{\delta}$). More precisely:

\begin{Def}
For $0 \leq m \leq k$ and $0 \leq l \leq m$, we define the $(m,l)$-topological entropy by
\begin{equation*}
\begin{split}
h_{(m,l)}^{top}(f)=& \overline{\lim_{\delta \to 0}} \overline{\lim_n} \frac{1}{n} \log \sup_{\Delta \in X_m^{\delta} } ( \max \# E \mbox{ , }\\
&E \mbox{   } (n,\delta) \mbox{-separated   } E \subset X_l^{\delta , n} \mbox{   and   } \forall W \in E \mbox{   we have   } W \subset f^{-n}(\Delta) ).  \\
\end{split}
\end{equation*}
\end{Def}
 
\begin{Rem}
For $m=0$ it is exactly the pointwise preimage entropy $h_m(f)$ defined in \cite{Hu} (see \cite{NiPr} and \cite{FFN} too) naturally extended to the meromorphic maps context.

If we consider $h_{top}^l(f)$ defined in \cite{DetVig}, we have clearly $h_{(k,l)}^{top}(f) \leq h_{top}^l(f)$, but the equality is not clear.
 
\end{Rem}

We give now some properties of these dimensional preimage entropies. We postpone their proofs to the paragraph \ref{properties}.

The first one gives the link between the new entropies and the usual topological entropy:

\begin{Pro}{\label{topological}}
We have $h_{(k,0)}^{top}(f)=h_{top}(f)$.
\end{Pro}

The both next Propositions give the behavior of these entropies when we vary the dimensions:

\begin{Pro}

For $0 \leq m \leq k$ and $1 \leq l \leq m$, we have

$$h_{(m,l-1)}^{top}(f) \geq h_{(m,l)}^{top}(f).$$

\end{Pro}

\begin{Pro}
For $0 \leq m \leq k-1$ and $0 \leq l \leq m$, we have

$$h_{(m+1,l)}^{top}(f) \geq h_{(m,l)}^{top}(f).$$

\end{Pro}

Now, in the same spirit as the Ruelle's inequality or \cite{Det1}, we give some inequalities between the dimensional preimage entropies and the Lyapounov exponents of suitable measures. 

Let $\mathcal{C}$ be the critical set of $f$, $I$ its indeterminacy set and $\mathcal{A}= I \cup \mathcal{C}$. Notice that when an ergodic invariant probability measure $\mu$ satisfies $\int \log d(x, \mathcal{A}) d \mu(x) > - \infty$, it means that its Lyapounov exponents are finite. In this paper, under this natural assumption, we prove the three following inequalities:

\begin{Thm}{\label{th1}}

Let $\mu$ be an ergodic invariant probability measure such that $\int \log d(x, \mathcal{A}) d \mu(x) > - \infty$ and suppose that its Lyapounov exponents satisfy

$$\chi_1 \geq \cdots \geq \chi_s > 0 \geq  \chi_{s+1} \geq \cdots \geq \chi_k.$$

Then $h_{(k-s,k-s)}^{top}(f) \geq h_{\mu}(f)$.

\end{Thm}

\begin{Thm}{\label{th2}}

Let $\mu$ be an ergodic invariant probability measure such that $\int \log d(x, \mathcal{A}) d \mu(x) > - \infty$ and suppose that its Lyapounov exponents satisfy

$$\chi_1 \geq \cdots \geq \chi_s > 0 \geq  \chi_{s+1} \geq \cdots \geq \chi_{s+l_0}= \cdots = \chi_{s+l_1} >  \chi_{s+l_1+1} \geq \cdots \geq \chi_k.$$

Then 

$$h_{(k-s-l_1,k-s-l_1)}^{top}(f) \geq h_{\mu}(f)+ 2 \chi_{s+1} + \cdots + 2 \chi_{s+l_1}.$$

\end{Thm}

And finally,

\begin{Thm}{\label{th3}}

Let $\mu$ be an ergodic invariant probability measure such that $\int \log d(x, \mathcal{A}) d \mu(x) > - \infty$. Denote by $\chi_1 , \cdots ,\chi_k$ its Lyapounov exponents. Then

$$h_{(0,0)}^{top}(f) \geq h_{\mu}(f)+ 2 \chi_{s+1} + \cdots + 2 \chi_{k}.$$

\end{Thm}

We can see the last Theorem as a Ruelle inverse formula. Indeed, when $f$ is invertible, this is exactly the Ruelle's inequality for $f^{-1}$ (because in this case $h_{(0,0)}^{top}(f)=0$).

Remark that a large part of this paper can be extended to real dynamical systems. We deal here with dominating meromorphic maps by interest to the indeterminacy set on which $f$ is not an application.

Here is the plan of this paper: in the first paragraph we prove the properties given for the dimensional preimage entropies. In the second one we recall some facts on Pesin's Theory and on the graph transform Theorem which are useful in the proofs of Theorems 1,2 and 3, and finally we prove these three Theorems.

\section{\bf Proofs of the properties}{\label{properties}}

We recall the properties and then we give their proofs.

\begin{Pro}
We have $h_{(k,0)}^{top}(f)=h_{top}(f)$.
\end{Pro}

\begin{proof}

Recall that $\displaystyle{h_{top}(f)= \lim_{\delta \to 0} \overline{\lim_n} \frac{1}{n} \log \sup (\max \# E \mbox{ , } E \mbox{   } (n,\delta) \mbox{-separated   })}$.

First we have clearly $h_{(k,0)}^{top}(f)\leq h_{top}(f)$.

Take $\epsilon >0$. We choose $\delta >0$ small enough so that

$$h_{top}(f)- \epsilon \leq \overline{\lim_n} \frac{1}{n} \log \sup (\max \# E \mbox{ , } E \mbox{   } (n,\delta) \mbox{-separated   }).$$

For $n_0 \in \Nn$, there exists $n \geq n_0$ with

$$h_{top}(f)- 2\epsilon \leq  \frac{1}{n} \log \sup (\max \# E \mbox{ , } E \mbox{   } (n,\delta) \mbox{-separated   }).$$

So, consider a set $E$ $(n,\delta)$-separated with $\# E \geq e^{h_{top}(f)n- 2\epsilon n}$.

$X$ is covered by a finite number $C$ of charts. We consider a chart with a number higher or equal to $\frac{1}{C}  e^{h_{top}(f)n- 2\epsilon n}$ of points of $f^n(E)$ (counted with multiplicity) and we take a subdivision of this chart into identical cubes of size $\delta$. Up to a multiplicative constant which depends only on $X$, one of these cubes, $\Delta$, contains at least $\frac{\delta^{2k}}{C}  e^{h_{top}(f)n- 2\epsilon n}$ points of $f^n(E)$.

This $\Delta$ is in $X_k^{\delta}$.

We proved that for all $n_0\in \Nn$, there exits $n \geq n_0$ such that

$$\frac{1}{n} \log  \sup_{\Delta \in X_k^{\delta} } (\max \# E \mbox{ , } E \mbox{   } (n,\delta) \mbox{-separated   } E \subset f^{-n}(\Delta) ) \geq h_{top}(f)- 2\epsilon + \frac{1}{n} \log \frac{\delta^{2k}}{C} $$

because if $f^n(x) \in \Delta$ then $x \in f^{-n}(\Delta)$.

Now, take the $\displaystyle \overline{\lim_n}$, the $\displaystyle \overline{\lim_{\delta \to 0}}$ and then the result follows by letting $\epsilon \to 0$.

\end{proof}

\begin{Pro}

For $0 \leq m \leq k$ and $1 \leq l \leq m$, we have

$$h_{(m,l-1)}^{top}(f) \geq h_{(m,l)}^{top}(f).$$

\end{Pro}

\begin{proof}

Take $\epsilon >0$. We choose $\delta >0$ small enough so that

\begin{equation*}
\begin{split}
& \overline{\lim_n} \frac{1}{n} \log \sup_{\Delta \in X_m^{\delta} } ( \max \# E \mbox{ , } E \mbox{   } (n,\delta) \mbox{-separated   } E \subset X_l^{\delta , n} \mbox{   and   } \forall W \in E \mbox{   ,   } W \subset f^{-n}(\Delta) )  \\
& \geq h_{(m,l)}^{top}(f) - \epsilon. \\
\end{split}
\end{equation*}

For $n_0 \in \Nn$, there exists $n \geq n_0$ with

\begin{equation*}
\begin{split}
&  \frac{1}{n} \log \sup_{\Delta \in X_m^{\delta} } ( \max \# E \mbox{ , } E \mbox{   } (n,\delta) \mbox{-separated   } E \subset X_l^{\delta , n} \mbox{   and   } \forall W \in E \mbox{   ,   } W \subset f^{-n}(\Delta) )  \\
& \geq h_{(m,l)}^{top}(f) - 2 \epsilon. \\
\end{split}
\end{equation*}

We take $\Delta$ such that

\begin{equation*}
\begin{split}
&  \frac{1}{n} \log ( \max \# E \mbox{ , } E \mbox{   } (n,\delta) \mbox{-separated   } E \subset X_l^{\delta , n} \mbox{   and   } \forall W \in E \mbox{   ,   } W \subset f^{-n}(\Delta) )  \\
& \geq h_{(m,l)}^{top}(f) - 3 \epsilon, \\
\end{split}
\end{equation*}

and we consider $E$ which maximizes the previous quantity.

Let $W \in E$. By definition, $W= \{ (X,\Phi(X)) \mbox{   ,   } X \in B_l(y, e^{- \delta n}) \}$. Here $y=(y_1, \dots , y_l)$ and we consider

$$W'= \{ (x_1, \cdots , x_{l-1}, y_l, \Phi(x_1 , \cdots , x_{l-1}, y_l)) \mbox{   ,   } (x_1, \cdots , x_{l}) \in B_l(y, e^{- \delta n}) \cap (x_l=y_l) \}.$$

In what follows, we identify $X'=(x_1, \cdots , x_{l-1})$ with $(x_1, \cdots , x_{l-1}, y_l)$.

We have  $B_l(y, e^{- \delta n}) \cap (x_l=y_l)=B_{l-1}(y, e^{- \delta n})$ and 

$$W'= \{ (X',\Psi(X')) \mbox{   ,   } X' \in B_{l-1}(y, e^{- \delta n}) \}$$

with $\Psi(X')=(y_l, \Phi(x_1 , \cdots , x_{l-1}, y_l))$.

Remark that $Lip(\Psi) \leq Lip(\Phi) \leq 1$ and $f^k(W') \subset f^k(W) \subset X \setminus I$ for $k=0, \cdots n-1$. So $W'$ is in $ X_{l-1}^{\delta , n}$.

We use the same method with the others $W \in E$ and we obtain a set $E'$ in $ X_{l-1}^{\delta , n}$ which is $(n, \delta)$-separated (because $W' \subset W$ and $E$ is $(n, \delta)$-separated). Moreover all the $W'$ are contained in $f^{-n}(\Delta)$. So,

\begin{equation*}
\begin{split}
&  \frac{1}{n} \log ( \max \# E' \mbox{ , } E' \mbox{   } (n,\delta) \mbox{-separated   } E' \subset X_{l-1}^{\delta , n} \mbox{   and   } \forall W \in E' \mbox{   ,   } W \subset f^{-n}(\Delta) )  \\
& \geq h_{(m,l)}^{top}(f) - 3 \epsilon, \\
\end{split}
\end{equation*}

and then

\begin{equation*}
\begin{split}
&  \frac{1}{n} \log \sup_{\Delta \in X_m^{\delta} } ( \max \# E' \mbox{ , } E' \mbox{   } (n,\delta) \mbox{-separated   } E' \subset X_{l-1}^{\delta , n} \mbox{   and   } \forall W \in E' \mbox{   ,   } W \subset f^{-n}(\Delta) )  \\
& \geq h_{(m,l)}^{top}(f) - 3 \epsilon. \\
\end{split}
\end{equation*}

Now, take the $\displaystyle \overline{\lim_n}$, the $\displaystyle \overline{\lim_{\delta \to 0}}$ and then the result follows by letting $\epsilon \to 0$.

\end{proof}

\begin{Pro}
For $0 \leq m \leq k-1$ and $0 \leq l \leq m$, we have

$$h_{(m+1,l)}^{top}(f) \geq h_{(m,l)}^{top}(f).$$

\end{Pro}

\begin{proof}

Consider $\Delta \in  X_m^{\delta}$. $\Delta$ is the graph of a holomorphic map $\Phi$ over $B_m(y, \delta)$ with $Lip(\Phi) \leq 1$.

$B_m(y, \delta)$ is contained in a complex plane $P$ of dimension $m$. Take an orthogonal vector $v$ to $P$ and consider $P \oplus \Cc v$. We obtain a complex plane of dimension $m+1$ and inside we construct $B_{m+1}((y,0), \delta)$.

For $X'=(x_1,\cdots , x_{m+1}) \in B_{m+1}((y,0), \delta)$, we define

$$\Psi(X')=(\Phi_2(x_1, \cdots , x_m), \cdots , \Phi_{k-m}(x_1, \cdots , x_m))$$

where $\Phi(X)=\Phi(x_1, \cdots , x_m)=(\Phi_1(x_1, \cdots , x_m), \cdots , \Phi_{k-m}(x_1, \cdots , x_m))$.

We have $Lip(\Psi) \leq Lip(\Phi) \leq 1$ and 

$$\Psi(y,0)=(\Phi_2(y_1, \cdots , y_m), \cdots , \Phi_{k-m}(y_1, \cdots , y_m))=0$$

because $\Phi(y)=0$. It implies that $\Delta'=\{(X', \psi(X')) \mbox{   ,   } X' \in B_{m+1}((y,0), \delta) \} $ is in $X_{m+1}^{\delta}$.

Moreover, $\Delta \subset \Delta'$:

Indeed, take a point $a \in \Delta$. We have $a=(x_1, \cdots ,x_m,\Phi_1(x_1, \cdots , x_m), \cdots , \Phi_{k-m}(x_1, \cdots , x_m) )$ with $(x_1, \cdots ,x_m) \in B_m(y, \delta)$. So

\begin{equation*}
\begin{split}
| \Phi_1(x_1, \cdots ,x_m)|&=| \Phi_1(x_1, \cdots ,x_m)-\Phi_1(y_1, \cdots ,y_m)| \\
& \leq 1 \times \|( x_1, \cdots ,x_m)-(y_1, \cdots ,y_m) \| \leq \delta.
\end{split}
\end{equation*}

Finally $a=(X',\Psi(X'))$ with $X'=(x_1, \cdots ,x_m,\Phi_1(x_1, \cdots , x_m)) \in B_{m+1}((y,0), \delta)$ and so $a \in \Delta'$.

Now, if we have a $(n, \delta)$-separated set $E \subset X_l^{\delta , n}$ such that $\forall W \in E \mbox{   ,   } W \subset f^{-n}(\Delta) $, we have $W \subset f^{-n}(\Delta')$ and then

$$h_{(m+1,l)}^{top}(f) \geq h_{(m,l)}^{top}(f).$$

\end{proof}

\begin{Rem}

For $l=0, \cdots, k$, we have $h_{(k,l)}^{top}(f) \leq h^l_{top}(f)$ where $h^l_{top}(f)$ is the $l$-topological entropy defined in \cite{DetVig}. The equality is not clear: it is not possible in general to put the $f^n(W_i)$ ($W_i$ graphs over a ball $B_l(x, e^{-\delta n})$) in a same $\Delta$ of dimension $k$ like we have done in the proof of Proposition \ref{topological}.

\end{Rem}

\section{\bf{Proofs of the Theorems}}

\subsection{{\bf Pesin's theory and graph transform Theorem}}{\label{paragraphePesin}}

In this paragraph we recall the Pesin's theory and the graph transform Theorem. We follow \cite{Det1} and \cite{DetNgu}.

\subsubsection{{\bf Oseledets' Theorem and Pesin's theory}}{\label{Pesin}}

Recall that $X$ is a compact complex manifold with dimension $k$ and $f: X \rightarrow X$ a meromorphic dominant map. 

We take a family of charts $(\tau_x)_{x \in X}$ which satisfy $\tau_x(0)=x$, $\tau_x$ is defined on the ball $B(0, \epsilon_0) \subset \Cc^k$ with $\epsilon_0$ independant of $x$ and the norm of the derivatives of order one and two of $\tau_x$ on $B(0, \epsilon_0)$ is bounded by above by a constant independant of $x$. To construct these charts, take a finite family $(U_i, \psi_i)$ of charts of $X$ and compose them by tranlations.

Let $\mu$ be an ergodic invariant probability measure such that  $\int \log d(x,\mathcal{A}) d \mu(x) > - \infty$. Define $\Omega=X \setminus \cup_{i \geq 0} f^{-i}(\mathcal{A})$. The measure $\mu$ has no mass on $\mathcal{A}$ and is $f$-invariant, so $\mu$ is a probability on $\Omega$.

We define the natural extension:

$$\widehat{\Omega}:= \{ \widehat{x}=( \cdots, x_0, \cdots , x_n , \cdots) \in \Omega^{\Zz} \mbox{ , } f(x_{n})=x_{n+1} \}.$$

In this space, $f$ induces a map $\widehat{f}$ which is the left-shift. If $\pi$ is the projection $\pi(\widehat{x})=x_0$, then there exists an unique probability measure $\widehat{\mu}$ invariant by $\widehat{f}$ which satisfies $\pi_{*} \widehat{\mu}=\mu$.

Consider $f_x= \tau_{f(x)}^{-1} \circ f \circ \tau_x$ which is well defined in the neighbourhood of $0$ when $x$ is outside $I$. The cocycle at which we apply Pesin's theory is 
\begin{equation*}
\begin{split}
A :\ & \widehat{\Omega} \longrightarrow M_k(\Cc)\\
& \widehat{x} \longrightarrow Df_x(0)\\
\end{split}
\end{equation*}

where $M_k(\Cc)$ is the set of $k \times k$ square matrices with complex coefficients and $\pi(\widehat{x})=x$. By using the hypothesis $\int \log d(x,\mathcal{A}) d \mu(x) > - \infty$, we have (see lemma 7 in \cite{Det1})

$$\int \log^+ \| (A( \widehat{x}))^{\pm 1} \| d \widehat{\mu} ( \widehat{x}) < + \infty.$$

So, we can use Oseledets' Theorem: 

\begin{Thm}{\label{oseledets}}

 There exits real numbers $\lambda_1 > \lambda_2 > \cdots > \lambda_{p_0} > - \infty$, some integers $m_1, \cdots, m_{p_0}$ and a set $\widehat{\Gamma}$ with full measure for $\widehat{\mu}$ such that for $\widehat{x} \in \widehat{\Gamma}$ we have $\Cc^k= \bigoplus_{i=1}^{p_0} E_i(\widehat{x})$ where $E_i(\widehat{x})$ are vector subspaces of dimension $m_i$ which verify:

1) $A(\widehat{x}) E_i(\widehat{x}) = E_i(\widehat{f}(\widehat{x}))$.

2) For $v \in E_i(\widehat{x}) \setminus \{0 \}$, we have 
$$\lim_{n \rightarrow \pm \infty} \frac{1}{n} \log \|A(\widehat{f}^{n-1}(\widehat{x})) \cdots A(\widehat{x}) v \| = \lambda_i.$$

For all $\delta > 0$, there exists a function $C_{\delta} :  \widehat{\Gamma} \longrightarrow GL_k(\Cc)$ such that for $\widehat{x} \in \widehat{\Gamma}$:

1) $\lim_{n \rightarrow \infty} \frac{1}{n} \log \| C^{\pm 1}_{\delta} (\widehat{f}^n(\widehat{x})) \|=0$ (tempered function).

2) $C_{\delta}(\widehat{x})$ sends the canonical decomposition $\bigoplus_{i=1}^{p_0} \Cc^{m_i}$ on $\bigoplus_{i=1}^{p_0} E_i(\widehat{x})$.

3) $A_{\delta}(\widehat{x})= C_{\delta}^{-1}(\widehat{f}(\widehat{x})) A(\widehat{x}) C_{\delta}(\widehat{x})$ is a diagonal block matrix $diag(A^1_{\delta}(\widehat{x}), \cdots, A^{p_0}_{\delta}(\widehat{x}))$ where each $A^{i}_{\delta}(\widehat{x})$ is a square $m_i \times m_i$ matrix and
$$\forall v \in \Cc^{m_i}   \mbox{    } \mbox{  we have   } \mbox{    } e^{\lambda_i - \delta} \|v\| \leq \| A^{i}_{\delta}(\widehat{x}) v \| \leq e^{\lambda_i + \delta} \|v\|.$$

\end{Thm}

Denote by $g_{\widehat{x}}$ the function 

$$g_{\widehat{x}}= C_{\delta}^{-1}(\widehat{f}(\widehat{x}))  \circ f_x \circ C_{\delta}(\widehat{x})$$

where $\pi(\widehat{x})=x$. 

We will use the following proposition (see \cite{DetNgu} Proposition 1.1):

\begin{Pro}{\label{prop1}}

There exists a set $\widehat{\Gamma'}$ with full measure for $\widehat{\mu}$ and a measurable map $r_1 : \widehat{\Gamma'} \rightarrow ]0, 1]$ such that for all $\widehat{x} \in \widehat{\Gamma'}$ we have $e^{- \delta} \leq \frac{r_1(\widehat{f}(\widehat{x}))}{r_1(\widehat{x})} \leq e^{\delta}$ and

1) $g_{\widehat{x}}(0)=0$.

2) $D g_{\widehat{x}}(0)=A_{\delta}(\widehat{x})=diag(A^1_{\delta}(\widehat{x}), \cdots, A^{p_0}_{\delta}(\widehat{x}))$.

3) $g_{\widehat{x}}(w)$ is holomorphic for $\| w \| \leq r_1(\widehat{x})$ and $\|D^2 g_{\widehat{x}}(w)\| \leq  \frac{1}{r_1(\widehat{x})}$ for $\| w \| \leq r_1(\widehat{x})$.

In particular, if we consider $g_{\widehat{x}}(w)=D g_{\widehat{x}}(0)w + h(w)$, we have $\| Dh(w) \| \leq \frac{1}{r_1(\widehat{x})} \|w\|$ for $\| w \| \leq r_1(\widehat{x})$.

\end{Pro}

We have the same type of proposition for $g_{\widehat{x}}^{-1}$:

\begin{Pro}{\label{prop2}}

There exists a set $\widehat{\Gamma''}$ with full measure for $\widehat{\mu}$ and a measurable map $r_2 : \widehat{\Gamma'} \rightarrow ]0, 1]$ such that for all $\widehat{x} \in \widehat{\Gamma''}$ we have $e^{- \delta} \leq \frac{r_2(\widehat{f}(\widehat{x}))}{r_2(\widehat{x})} \leq e^{\delta}$ and

1) $g^{-1}_{\widehat{x}}(0)=0$.

2) $Dg^{-1}_{\widehat{x}}(0)=(D g_{\widehat{x}}(0))^{-1}=A_{\delta}^{-1}(\widehat{x})$=$diag((A^1_{\delta}(\widehat{x}))^{-1}, \cdots, (A^{p_0}_{\delta}(\widehat{x}))^{-1})$.

3) $g^{-1}_{\widehat{x}}(w)$ is holomorphic for $\| w \| \leq r_2(\widehat{x})$ and $\|D^2 g^{-1}_{\widehat{x}}(w)\| \leq  \frac{1}{r_2(\widehat{x})}$ for $\| w \| \leq r_2(\widehat{x})$.

In particular, if we consider $g^{-1}_{\widehat{x}}(w)=D g^{-1}_{\widehat{x}}(0)w + h(w)$, we have $\| Dh(w) \| \leq \frac{1}{r_2(\widehat{x})} \|w\|$ for $\| w \| \leq r_2(\widehat{x})$.

\end{Pro}

\begin{proof}

We use the proposition 9 in \cite{Det1}. Notice that $g^{-1}_{\widehat{x}}$ here corresponds to $g^{-1}_{\widehat{f}(\widehat{x})}$ in \cite{Det1}.

So for $\widehat{x} \in \widehat{\Gamma}$, $g^{-1}_{\widehat{x}}(w)$ is defined and holomorphic for $\|w\| \leq 2 \epsilon_0' \frac{d(x_0, \mathcal{A})^p}{\| C_{\delta}(\widehat{f}(\widehat{x})) \|}$, $g^{-1}_{\widehat{x}}(0)=0$ and  $Dg^{-1}_{\widehat{x}}(0)=(D g_{\widehat{x}}(0))^{-1}=A_{\delta}^{-1}(\widehat{x})$ (see the end of page 99 in \cite{Det1}).

We have too:

$$\| D f^{-1}_{x_0} (w) \| + \| D^2 f^{-1}_{x_0} (w) \| \leq \tau d(x_0, \mathcal{A})^{-p'}$$

for $\|w\| \leq \epsilon_0' d(x_0, \mathcal{A})^p$.

Now, 

$$D g^{-1}_{\widehat{x}}(w)= C^{-1}_{\delta}(\widehat{x}) \circ  D f^{-1}_{x_0}  (C_{\delta}(\widehat{f}(\widehat{x}))(w)) \circ C_{\delta}(\widehat{f}(\widehat{x}))$$

which gives

$$\| D^2 g^{-1}_{\widehat{x}}(w) \| = \|C^{-1}_{\delta}(\widehat{x})\| \| D^2 f^{-1}_{x_0}(C_{\delta}(\widehat{f}(\widehat{x}))(w))\|  \|  C_{\delta}(\widehat{f}(\widehat{x}) \|^2 .$$

So for $\|w\| \leq \epsilon_0' \frac{d(x_0, \mathcal{A})^p}{\| C_{\delta}(\widehat{f}(\widehat{x})) \|}$,

$$
\| D^2 g^{-1}_{\widehat{x}}(w) \| \leq  \tau d(x_0, \mathcal{A})^{-p'} \|C^{-1}_{\delta}(\widehat{x})\|   \|  C_{\delta}(\widehat{f}(\widehat{x})) \|^2. $$

We define

$$\alpha(\widehat{x})= \max \left( 1 ,\frac{\| C_{\delta}(\widehat{f}(\widehat{x})) \|}{ \epsilon_0' d(x_0, \mathcal{A})^p}, \tau d(x_0, \mathcal{A})^{-p'} \|C^{-1}_{\delta}(\widehat{x})\|   \|  C_{\delta}(\widehat{f}(\widehat{x})) \|^2 \right).$$

Since $\int \log d(\pi(\widehat{x}), \mathcal{A}) d \widehat{\mu}( \widehat{x})= \int \log d(x, \mathcal{A}) d \mu(x) > - \infty$, the Birkhoff's Theorem implies that there exists a set $\widehat{\Gamma}'' \subset \Gamma$ with full measure for $ \widehat{\mu}$ on which

$$\lim_{n \to + \infty} \frac{1}{n} \sum_{i=0}^{n-1} \log d(\pi(\widehat{f}^{i} (\widehat{x})), \mathcal{A})= \lim_{n \to + \infty} \frac{1}{n} \sum_{i=0}^{n-1} \log d(\pi(\widehat{f}^{-i} (\widehat{x})), \mathcal{A})  >- \infty.$$

So for $\widehat{x} \in \widehat{\Gamma}''$, we have

$$\lim_{n \to \pm \infty} \frac{1}{n} \log d(\pi(\widehat{f}^{n} (\widehat{x})), \mathcal{A})=0.$$

Since $\|C^{\pm}_{\delta}(\widehat{x})\|$ are tempered functions, $\alpha(\widehat{x})$ inherits the same property.

If we use the lemma S.2.12 in \cite{KH}, there exists a function $\alpha_{\delta} \geq \alpha$ such that $\e^{- \delta} \leq \frac{\alpha_{\delta}(\widehat{f}(\widehat{x}))}{\alpha_{\delta}(\widehat{x})} \leq e^{\delta}$ for all $\widehat{x} \in \widehat{\Gamma}''$.

Define $r_2(\widehat{x})= \frac{1}{\alpha_{\delta}(\widehat{x})}$ (which is $\leq 1$). We have $\e^{- \delta} \leq \frac{r_2(\widehat{f}(\widehat{x}))}{r_2(\widehat{x})} \leq e^{\delta}$ for all $\widehat{x} \in \widehat{\Gamma}''$, the function $g^{-1}_{\widehat{x}}(w)$ is defined and holomorphic for 

$$\|w\| \leq r_2(\widehat{x})= \frac{1}{\alpha_{\delta}(\widehat{x})} \leq \frac{1}{\alpha(\widehat{x})} \leq \epsilon_0' \frac{d(x_0, \mathcal{A})^p}{\| C_{\delta}(\widehat{f}(\widehat{x})) \|}$$

and finally

$$\| D^2 g^{-1}_{\widehat{x}}(w) \| \leq \alpha(\widehat{x}) \leq \alpha_{\delta}(\widehat{x}) = \frac{1}{r_2(\widehat{x})}$$

for $\|w\| \leq r_2(\widehat{x})$.

\end{proof}

\subsection{{\bf Graph transform Theorem}}

Consider $\Cc^k$ endowed of the norm $\|x\|=\max(|x_1|, \cdots , |x_k|)$. Define $B_l(0,R)$ the ball with center $0$ and radius $R$ in $\Cc^l$. Consider

$$g(X,Y)=(g_1(X,Y),g_2(X,Y))=(AX + R(X,Y), BY + U(X,Y))$$

with $(X,0) \in E_1$, $(0,Y) \in E_2$ and $A: \Cc^{k_1} \longrightarrow \Cc^{k_1}$, $B: \Cc^{k_2} \longrightarrow \Cc^{k_2}$ linear maps with $k=k_1 + k_2$. We suppose that $g: B_k(0,R_0) \longrightarrow B_k(0, R_1)$ is holomorphic with $R_0 \leq R_1$, $g(0)=0$ and $\max(\|DR(Z)\|, \|DU(Z)\|) \leq \gamma$ on $B_k(0,R_0)$. 

We suppose $A$ invertible, $\|B\| < \|A^{-1}\|^{-1}$ and we define $\xi=1 - \|B\| \|A^{-1}\| \in ]0,1]$.

We use the following version of the graph transform Theorem (see the paragraph 4 in \cite{Det1} and the Theorem 3 in \cite{DetNgu}) : 

\begin{Thm}{\label{graph}}

Let $\{(X, \phi(X)), X \in D\}$ be a graph in $B_k(0,R_0)$ over a part $D$ in $E_1$ which verifies $Lip(\phi)\leq \gamma_0 \leq 1$. 

If $\gamma \|A^{-1}\| (1 + \gamma_0)<1$ then the image by $g$ of this graph is a graph over $\pi_0(g( \mbox{graph of } \phi))$ where $\pi_0$ is the projection on $E_1$. Moreover, if $(X, \psi(X))$ is this new graph, we have:

$$\| \psi(X_1) - \psi(X_2) \| \leq \frac{\| B \| \gamma_0 + \gamma( 1 + \gamma_0)}{\|A^{-1}\|^{-1} - \gamma(1 + \gamma_0)}\|X_1 - X_2\|$$

which is smaller than  $\gamma_0 \|X_1 - X_2\|$ if $\gamma \leq \epsilon(\gamma_0, \xi)$.

Finally, if $B_{k_1}(0, \alpha) \subset D$ and $\| \phi(0)\| \leq \beta$, then $\pi_0(g( \mbox{graph of } \phi))$ contains $B_{k_1}(0, (\|A^{-1}\|^{-1} - \gamma(1 + \gamma_0)) \alpha  - \gamma \beta )$ and $\| \Psi(0) \| \leq (1 + \gamma_0)(\| B \| \beta + \gamma \beta + \| D^2g\|_{B_k(0,R_0)} \beta^2)$ (if $\gamma \leq \epsilon(\gamma_0, \xi)$).

\end{Thm}

\subsection{\bf{Proof of the Theorem \ref{th2}}}

Let $\mu$ be an ergodic invariant probability with $\int \log d(x, \mathcal{A}) d \mu(x) > - \infty$ and suppose that its Lyapounov exponents satisfy

$$\chi_1 \geq \cdots \geq \chi_s > 0 \geq  \chi_{s+1} \geq \cdots \geq \chi_{s+l_0}= \cdots = \chi_{s+l_1} >  \chi_{s+l_1+1} \geq \cdots \geq \chi_k.$$

We can suppose that $h_{\mu}(f)+ 2 \chi_{s+1} + \cdots + 2 \chi_{s+l_1} > 0$, otherwise there is nothing to do.

Here is the plan of the proof: in the first part we contruct points $x_1, \cdots , x_{N'}$ which are $(n, \delta)$-separated and with good properties. In the second part we give the contruction of $\Delta \in X_{k-s-l_1}^{ \delta}$ which occurs in the definition of $h_{(k-s-l_1,k-s-l_1)}^{top}(f)$. Finally, we construct the $W \subset f^{-n}(\Delta)$ to obtain the Theorem.  

\subsubsection{Construction of $x_1, \cdots , x_{N'}$}

The Lyapounov exponents do not depend on the choice of the charts of $X$, so we can suppose that $\mu$ has no mass on the boundaries of charts $(U_i, \psi_i)$.

In particular, if we choose $\epsilon_1 >0$ small enough, the mass for $\mu$ of a $\epsilon_1$-neighbourhood $V_{\epsilon_1}$ of the boundaries of these charts is smaller than $\frac{1}{10}$.

Define $d_n(x,y)=\displaystyle \max_{0 \leq i \leq n-1} dist(f^{i}(x), f^{i}(y))$ and $B_n(x, \delta)$ the ball with center $x$ and radius $\delta$ for the metric $d_n$.

By using the Brin-Katok's theorem, we have 

$$h_{\mu}(f)= \lim_{\delta \to 0} \liminf_{n \to + \infty} - \frac{1}{n} \log \mu B_n(x , \delta)$$

for $\mu$ almost every $x$.

Take $\epsilon >0$ and define

$$\Lambda_{\delta \mbox{,} n}= \{ x \mbox{  ,  } \mu B_n(x, 6 \delta) \leq \e^{- h_{\mu}(f)n + \epsilon n} \}.$$

If $\delta$ is small enough, we have

$$\frac{9}{10} \leq \mu( \{ x \mbox{  ,  } \liminf_{n \to + \infty} - \frac{1}{n} \log \mu B_n(x , 6 \delta)  \geq h_{\mu}(f) - \frac{\epsilon}{2} \} ) \leq \mu \left( \cup_{n_0} \cap_{n \geq n_0} \Lambda_{\delta \mbox{,} n} \right).$$

So, if $n_0$ is high enough, we have $\displaystyle \mu \left( \cap_{n \geq n_0} \Lambda_{\delta \mbox{,} n} \right) \geq \frac{8}{10}$ and then we define $\Lambda_{n_0} = V_{\epsilon_1}^c \cap \left( \cap_{n \geq n_0} \Lambda_{\delta \mbox{,} n} \right)$. We obtain $\mu(\Lambda_{n_0}) \geq \frac{7}{10}$.

By using Lusin's Theorem, we can find a compact set $\widehat{\Gamma_0} \subset  \widehat{\Gamma} \cap \widehat{\Gamma'} \cap \widehat{\Gamma''}$ (where $\widehat{\Gamma}$, $\widehat{\Gamma'}$ and $\widehat{\Gamma''}$ are defined in the Theorem \ref{oseledets} and Propositions \ref{prop1} and \ref{prop2}) with mass $\geq \frac{9}{10}$ for $\widehat{\mu}$, such that $\widehat{x} \longrightarrow C_{\delta}^{\pm 1}(\widehat{x})$, $\widehat{x} \longrightarrow r_{1}(\widehat{x})$, $\widehat{x} \longrightarrow r_{2}(\widehat{x})$ are continuous on $\widehat{\Gamma_0}$.

For $\alpha_0 >0$ small enough, we have $\alpha_0 \leq \| C_{\delta}^{\pm 1}(\widehat{x}) \| \leq \frac{1}{\alpha_0}$, $r_1(\widehat{x}) \geq \alpha_0$ and $r_2(\widehat{x}) \geq \alpha_0$ on $\widehat{\Gamma_0}$.

Fix $n \geq n_0$ high with respect to $\delta$ and consider

$$\Lambda=\Lambda_{n_0} \cap f^{-n}(\Lambda_{n_0}) \cap \pi(\widehat{f}^{-n}(\widehat{\Gamma_0}) \cap \widehat{\Gamma_0}).$$

This set verifies $\mu(\Lambda) \geq 1-\frac{3}{10}-\frac{3}{10}-\frac{1}{10}-\frac{1}{10}=\frac{1}{5}$.

If $x \in \Lambda$, we have $\mu B_n(x, 6 \delta) \leq \e^{- h_{\mu}(f)n + \epsilon n} $ so we can find $x_1, \cdots, x_N \in X$ with $N \geq \frac{1}{5} \e^{ h_{\mu}(f)n - \epsilon n}$ which are $(n, 6 \delta)$-separated and such that $x_i=\pi( \widehat{x_i}) \in \Lambda_{n_0}$, $f^n(x_i)= \pi(\widehat{f}^n( \widehat{x_i})) \in \Lambda_{n_0}$ with $\widehat{x_i}, \widehat{f}^n( \widehat{x_i}) \in \widehat{\Gamma_0}$ (for $i=1, \cdots, N$).

Recall that we have

$$\chi_1 \geq \cdots \geq \chi_s > 0 \geq  \chi_{s+1} \geq \cdots \geq \chi_{s+l_0}= \cdots = \chi_{s+l_1} >  \chi_{s+l_1+1} \geq \cdots \geq \chi_k.$$

Denote by $E_1(\widehat{x}), \cdots ,E_m(\widehat{x})$ the $E_i(\widehat{x})$ of the Oseledets' Theorem which correspond to $\chi_1 , \cdots , \chi_{s+l_1}$ and $E_{m+1}(\widehat{x}), \cdots ,E_q(\widehat{x})$ the $E_i(\widehat{x})$ of $\chi_{s+l_1+1}, \cdots , \chi_k$.

Define $\displaystyle E^{u}(\widehat{x})= \oplus_{i=1}^{m} E_i(\widehat{x})$ and $\displaystyle E^{s}(\widehat{x})= \oplus_{i=m+1}^{q} E_i(\widehat{x})$. It will be usefull to decompose $E^{u}(\widehat{x})$ into $E^{u}(\widehat{x})=E_1^{u}(\widehat{x}) \oplus E_2^{u}(\widehat{x})$ where $E_1^{u}(\widehat{x})$ corresponds to the positive Lyapounov exponents.

We subdivide $X$ into cubes of size $\e^{-8 \delta n}$. There are $\e^{16k \delta n}$ such cubes (modulo a multiplicative constant which depends only on $X$).

We can find one of these cubes, $C$ which contains at least $N'= \frac{1}{5} \e^{ h_{\mu}(f)n - \epsilon n} \e^{-16k \delta n}$ points $f^n(x_i)$. To simplify the notations, we suppose that it corresponds to the indices $i=1, \dots , N'$.

Now, we want to find a complex plane $\Delta$ with dimension $k-s-l_1$ that we will pull-back (from $f^n(x_i)$ to $x_i$) to obtain "stables" manifolds in $f^{-n}(\Delta)$. The difficulty is to have a $\Delta$ which is common to a large number of $f^n(x_i)$.

\subsubsection{Construction of $\Delta$}

To obtain $\Delta$, there are two steps: in the first one, we construct "unstable" manifolds $W_n(f^n(x_i))$ ($i=1, \cdots , N'$) by using push-forward and graph transforms. Then, we bound by below the volume of these manifolds and we take $\Delta$ which intersects a large many number of them.

\medskip

{\bf Step 1 Construction of "unstable" manifolds $W_n(f^n(x_i))$ ($i=1, \cdots , N'$):}

\medskip

Fix $f^n(x_{i_0})$ one of the $f^n(x_i)$. Take $0< \gamma_0 <1$ with $ \frac{\gamma_0}{\alpha_0^2} \leq 1$.

For $i \in \{1, \cdots , N' \}$ we consider the coordinate system $C_{\delta}^{-1}(\widehat{x_i})E^{u}(\widehat{x_i}) \oplus C_{\delta}^{-1}(\widehat{x_i}) E^{s}(\widehat{x_i})$. 

In these coordinates, we begin with the graph $(X, \Phi_0(X))$ with $X \in B_{s}(0, \e^{-4 \delta n}) \times B_{l_1}(0, \e^{-8 \delta n}) \subset C_{\delta}^{-1}(\widehat{x_i})E_1^{u}(\widehat{x_i}) \oplus C_{\delta}^{-1}(\widehat{x_i}) E_2^{u}(\widehat{x_i})$ and $\Phi_0 \equiv 0$.

We push forward this graph by $g_{\widehat{f}^l(\widehat{x})}$ ($l=0, \cdots, n-1$), by using the following Lemma:

\begin{Lem}

Fix $l=0,\cdots , n-1$. In the coordinate system

$$C_{\delta}^{-1}(\widehat{f}^l(\widehat{x_i}))E^{u}(\widehat{f}^l(\widehat{x})) \oplus C_{\delta}^{-1}(\widehat{f}^l(\widehat{x_i})) E^{s}(\widehat{f}^l(\widehat{x}))$$

we consider a graph $(X, \Phi(X))$ over a part contained in $B_{s+l_1}(0, \e^{-4 \delta n}) \subset C_{\delta}^{-1}(\widehat{f}^l(\widehat{x_i}))E^{u}(\widehat{f}^l(\widehat{x}))$ with $Lip( \Phi) \leq \gamma_0$ and $\Phi(0)=0$. The image by $g_{\widehat{f}^l(\widehat{x_i})}$ is a graph $(X, \Psi(X))$ over a part of $C_{\delta}^{-1}(\widehat{f}^{l+1}(\widehat{x_i}))E^{u}(\widehat{f}^{l+1}(\widehat{x}))$
which satisfies $Lip( \Psi) \leq \gamma_0$.

\end{Lem}
 
\begin{proof}

We use the graph transform Theorem.

In the coordinates 

$$C_{\delta}^{-1}(\widehat{f}^l(\widehat{x_i}))E^{u}(\widehat{f}^l(\widehat{x})) \oplus C_{\delta}^{-1}(\widehat{f}^l(\widehat{x_i})) E^{s}(\widehat{f}^l(\widehat{x}))$$

we have

$$g_{\widehat{f}^l(\widehat{x_i})}(X,Y)=(AX+R(X,Y),BY+U(X,Y))$$

where $(A,B)=diag(A_{\delta}^1(\widehat{f}^l(\widehat{x_i})), \cdots , A_{\delta}^{p_0}(\widehat{f}^l(\widehat{x_i})))$, $\|B\| \leq \e^{\chi_{s+l_1+1} + \delta}$ and $\|A^{-1}\|^{-1} \geq \e^{\chi_{s+l_1} - \delta}$.

By using the Proposition \ref{prop1} we have

\begin{equation*}
\begin{split}
\max( \| DR(X,Y)\|, \|DU(X,Y)\|) &\leq \frac{1}{r_1(\widehat{f}^l(\widehat{x_i}))} \times \|(X,Y)\| \\
&\leq \frac{1}{\alpha_0 \e^{- \delta l}} 5 \e^{-4 \delta n} \leq \frac{5}{\alpha_0} \e^{-3 \delta n}\\
\end{split}
\end{equation*}

because $r_1$ is a tempered function and with $\|(X,Y)\| \leq R_0= 5 \e^{-4 \delta n}$.

Now, we verify the hypothesis of the graph transform Theorem. We have

$$\gamma \|A^{-1} \| (1+ \gamma_0) \leq \frac{5}{\alpha_0} \e^{-3 \delta n} \e^{- \chi_{s+l_1} + \delta} \times 2<1$$

$$ \frac{\|B\| \gamma_0 + \gamma(1+ \gamma_0)}{ \|A^{-1}\|^{-1} - \gamma(1+ \gamma_0)} \leq \frac{\e^{ \chi_{s+l_1+1} + \delta} \gamma_0 + \frac{5}{\alpha_0} \e^{-3 \delta n} \times 2}{\e^{ \chi_{s+l_1} - \delta}-\frac{5}{\alpha_0} \e^{-3 \delta n} \times 2} \leq \gamma_0$$

for $n$ high enough (take $\delta$ small enough to have $\e^{ \chi_{s+l_1+1} -\chi_{s+l_1}  + 2 \delta} < 1$).

By using the graph transform Theorem, the Lemma is proved.

\end{proof}

We begin with the graph $(X, \Phi_0(X))$ with $X \in B_{s}(0, \e^{-4 \delta n}) \times B_{l_1}(0, \e^{-8 \delta n}) \subset C_{\delta}^{-1}(\widehat{x_i})E_1^{u}(\widehat{x_i}) \oplus C_{\delta}^{-1}(\widehat{x_i}) E_2^{u}(\widehat{x_i})$ and $\Phi_0 \equiv 0$. By the previous Lemma, its image under $g_{\widehat{x_i}}$ is a graph $(X, \Phi_1(X))$ over a part of $C_{\delta}^{-1}(\widehat{f}(\widehat{x_i}))E^{u}(\widehat{f}(\widehat{x_i})))$ with $Lip(\Phi_1) \leq \gamma_0$ (and $\Phi_1(0)=0$).

We keep only the part over $B_{s+l_1}(0, \e^{-4 \delta n})$ (we do a cut-off) and we take now the image under $g_{\widehat{f}(\widehat{x_i})}$. We obtain a graph over a part of $C_{\delta}^{-1}(\widehat{f}^2(\widehat{x_i}))E^{u}(\widehat{f}^2(\widehat{x_i})))$, we do the cut-off and so on. At the end, wa have a graph $(X, \phi_n(X))$ over a part contained in $B_{s+l_1}(0, \e^{-4 \delta n}) \subset C_{\delta}^{-1}(\widehat{f}^n(\widehat{x_i}))E^{u}(\widehat{f}^n(\widehat{x_i}))$ with $Lip(\Phi_n) \leq \gamma_0$ (and $\Phi_n(0)=0$).

Denote by $W^0_n(f^n(x_i))$ these graphs. 

We want to considerate these graphs in the coordinate system 

$$C_{\delta}^{-1}(\widehat{f}^n(\widehat{x_{i_0}}))E^{u}(\widehat{f}^n(\widehat{x_{i_0}})) \oplus C_{\delta}^{-1}(\widehat{f}^n(\widehat{x_{i_0}})) E^{s}(\widehat{f}^n(\widehat{x_{i_0}}))$$.

It means to take the image by 

$$C=C_{\delta}^{-1}(\widehat{f}^n(\widehat{x_{i_0}})) \tau_{f^n(x_{i_0})}^{-1} \tau_{f^n(x_{i})} C_{\delta}(\widehat{f}^n(\widehat{x_{i}})).$$

We claim that this image is a graph $(X, \psi_n(X))$ over a part of $C_{\delta}^{-1}(\widehat{f}^n(\widehat{x_{i_0}}))E^{u}(\widehat{f}^n(\widehat{x_{i_0}}))$ with $Lip(\psi_n) \leq 2 \gamma_0$. To prove that, we follow \cite{DetNgu}:

We have $dist(f^n(x_i),f^n(x_{i_0})) \leq \e^{-8 \delta n} << \epsilon_1$, so $f^n(x_i)$ and $f^n(x_{i_0})$ are in the same chart. It implies that $\tau_{f^n(x_{i})}= \psi \circ t_1$ and $\tau_{f^n(x_{i_0})}= \psi \circ t_2$ where $t_1$ and $t_2$ are translations.

There exists thus a vector $a$ such that $\tau_{f^n(x_{i_0})}^{-1} \circ \tau_{f^n(x_{i})}(w)=w+a$ and we can write $C(w)=g_1(w) + C_{\delta}^{-1}(\widehat{f}^n(\widehat{x_{i_0}}))a$ where $g_1(w)=C_{\delta}^{-1}(\widehat{f}^n(\widehat{x_{i_0}})) C_{\delta}(\widehat{f}^n(\widehat{x_{i}}))w$.

We begin by taking the image of $(X, \Phi_n(X))$ by $g_1$. We have

$$g_1(w)=C_{\delta}^{-1}(\widehat{f}^n(\widehat{x_{i_0}})) C_{\delta}(\widehat{f}^n(\widehat{x_{i_0}}))w + C_{\delta}^{-1}(\widehat{f}^n(\widehat{x_{i_0}})) (C_{\delta}(\widehat{f}^n(\widehat{x_{i}}))- C_{\delta}(\widehat{f}^n(\widehat{x_{i_0}})))w$$

and

$$\| C_{\delta}^{-1}(\widehat{f}^n(\widehat{x_{i_0}})) (C_{\delta}(\widehat{f}^n(\widehat{x_{i}}))- C_{\delta}(\widehat{f}^n(\widehat{x_{i_0}}))) \| \leq \frac{1}{\alpha_0} \|C_{\delta}(\widehat{f}^n(\widehat{x_{i}}))- C_{\delta}(\widehat{f}^n(\widehat{x_{i_0}})) \|$$

(because $\widehat{f}^n(\widehat{x_{i_0}}) \in  \widehat{\Gamma_0}$).

But, $\widehat{x} \to C_{\delta}(\widehat{x})$ is uniformly continuous on the compact set $\widehat{\Gamma_0}$, so there exits a function $\epsilon(\eta)$ with $\displaystyle \lim_{\eta \to 0} \epsilon(\eta)=0$ and

$$\forall \mbox{  } \widehat{z}, \widehat{y} \in \widehat{\Gamma_0} \mbox{  ,  } dist(\widehat{z},\widehat{y}) \leq \eta \Rightarrow \frac{1}{\alpha_0} \| C_{\delta}(\widehat{z})- C_{\delta}(\widehat{y}) \| \leq \epsilon(\eta) .$$

This function depends only on $\widehat{\Gamma_0}$ and $C_{\delta}$. 

Suppose that $d(\widehat{f}^n(\widehat{x_{i}}), \widehat{f}^n(\widehat{x_{i_0}})) \leq \eta$. We can write $g_1(X,Y)=(AX+CY,BY+DX)$ with $\|A^{-1}\|^{-1} \geq 1- \epsilon(\eta)$, $\|C\| \leq \epsilon(\eta)$, $\|D\| \leq \epsilon(\eta)$ and $\|B\| \leq 1+ \epsilon(\eta)$.

If we take the notations of the graph transform Theorem, we obtain

$$\gamma \|A^{-1} \| (1+ \gamma_0) \leq \frac{\epsilon(\eta)}{1- \epsilon(\eta)} \times 2 <1$$

if $\eta$ is small enough and moreover

$$ \frac{ \|B\| \gamma_0 + \gamma(1+ \gamma_0)}{\|A^{-1}\|^{-1} - \gamma (1+\gamma_0)} \leq \frac{(1+ \epsilon(\eta)) \gamma_0 + 2 \epsilon(\eta) }{ (1- \epsilon(\eta)) - 2 \epsilon(\eta)} \leq 2 \gamma_0$$

if $\eta$ is small enough.

The graph transform Theorem implies that the image by $g_1$ of the graph $(X,\Phi_n(X))$ is a graph $(X, \Phi_n^1(X))$ with $Lip \Phi_n^1 \leq 2 \gamma_0$ and $\Phi_n^1(0)=0$.

It remains to translate this last graph by $C_{\delta}^{-1}(\widehat{f}^n(\widehat{x_{i_0}}))a$. If we write this vector $(a_1,a_2)$, the image of $(X, \Phi_n^1(X))$ by the translation is

$$(X+a_1, \Phi_n^1(X)+a_2)=(X', \Psi_n(X'))$$

with $X'=X+a_1$ and $ \Psi_n(X')= \Phi_n^1(X)+a_2= \Phi_n^1(X'-a_1)+a_2$. So, this is a graph over a part of $C_{\delta}^{-1}(\widehat{f}^n(\widehat{x_{i_0}}))E^{u}(\widehat{f}^n(\widehat{x_{i_0}}))$.

Moreover, we have $Lip \Psi_n = Lip \Phi_n^1 \leq 2 \gamma_0$ and

\begin{equation*}
\begin{split}
\| \Psi_n(0) \| &= \| \Phi_n^1(-a_1) + a_2 \| \leq  \| \Phi_n^1(-a_1) - \Phi_n^1(0)\| + \|a_2\| \\
&\leq 2 \gamma_0 \|a_1\| + \|a_2\| \leq (1+ 2 \gamma_0) \|a\|. \\
\end{split}
\end{equation*}

But $a=\overrightarrow{ \psi^{-1}(f^n(x_{i_0})) \psi^{-1}(f^n(x_{i}))}$, thus

$$\|a\|=\| \psi^{-1}(f^n(x_{i})) - \psi^{-1}(f^n(x_{i_0})) \| \leq C \e^{-4 \delta n}$$

with $C$ which depends only on $X$.

Finally it gives $\| \Psi_n(0) \| \leq 3C  \e^{-4 \delta n}$.

Now, we consider a cover of $\widehat{X}$ by balls of radius $\eta$. Since $\widehat{X}$ is a compact set, a finite number $N(\eta)$ is sufficient to cover $\widehat{X}$ (and $N(\eta)$ do not depends on $n$).

Up to changing $N'$ into $N''= \frac{N'}{N(\eta)}$, we can suppose that $\widehat{f}^n(x_i)$ and $\widehat{f}^n(x_{i_0})$ are in the same ball and so the hypothesis $d(\widehat{f}^n(\widehat{x_{i}}), \widehat{f}^n(\widehat{x_{i_0}})) \leq \eta$ that we have done is satisfied.

Denote by $W_n(f^n(x_i))$ ($i=1, \cdots , N''$) these graphs. The goal now is to bound by below their $2(s+l_1)$-volume by $\e^{2 \chi_{s+1} n + \cdots + 2 \chi_{s+l_1} n - 30 \delta n k}$.

\medskip

{\bf Step 2: Bound by below of the $2(s+l_1)$-volume of $W_n(f^n(x_i))$ ($i=1, \cdots , N''$)}

\medskip

We have $W^0_n(f^n(x_i))= C^{-1}(W_n(f^n(x_i)))$, thus by the coarea formula (see \cite{Fe} p.258) we obtain

\begin{equation*}
\begin{split}
vol_{2(s+l_1)}(W^0_n(f^n(x_i)))&= \int_{W_n(f^n(x_i))} \| \Lambda^{2(s+l_1)} D C^{-1} (z) \| d \mathcal{H}^{2(s+l_1)}(z)\\
& \leq C(\alpha_0) vol_{2(s+l_1)}(W_n(f^n(x_i)))
\end{split}
\end{equation*}

because $D C^{-1}(z)= C_{\delta}^{-1}(\widehat{f}^n(\widehat{x_{i}}))C_{\delta}(\widehat{f}^n(\widehat{x_{i_0}}))$.

So it is enough to bound by below the volume of $W^0_n(f^n(x_i))$.

For that, we use a foliation of $W^0_n(f^n(x_i))$ as in \cite{Det1}.

We start again with the first graph $(X, \Phi_0(X))$ over $B_{s}(0, \e^{-4 \delta n}) \times B_{l_1}(0, \e^{-8 \delta n})$ and we push-forward slices of it. We consider for that, $B_{s}(0, \e^{-4 \delta n}) \times \{a_{s+1} \} \times \cdots \times \{a_{s+l_1} \} \times \{0\}^{k-s-l_1}$ with $B_{s}(0, \e^{-4 \delta n}) \subset C_{\delta}^{-1}(\widehat{x_i})E_1^{u}(\widehat{x_i})$ and $(a_{s+1}, \cdots , a_{s+l_1}) \in B_{l_1}(0, \e^{-8 \delta n})$.

This set is a graph $(X', \xi_0(X'))$ over a part of $C_{\delta}^{-1}(\widehat{x_i})E_1^{u}(\widehat{x_i})$ 

(with $\xi_0 \equiv (a_{s+1}, \cdots , a_{s+l_1},0, \cdots , 0)$).

As previously, the image of this graph by $g_{\widehat{x_i}}$ is a graph $(X', \xi_1(X'))$ over a part of $C_{\delta}^{-1}(\widehat{f}(\widehat{x_i}))E_1^{u}(\widehat{f}(\widehat{x_i})))$ with $Lip(\xi_1) \leq \gamma_0$. Moreover we have:

\begin{Lem}{\label{lem2}}
The graph $(X', \xi_1(X'))$ is a graph at least over $B_s(0, \e^{-4 \delta n})$ and $\| \xi_1(0) \| \leq \e^{-8 \delta n +2 \delta}$.
\end{Lem}

\begin{proof}

We consider the graph transform Theorem (Theorem \ref{graph}) with $\alpha=\e^{-4 \delta n}$ and $\beta = \e^{-8 \delta n}$.

The projection of the graph $\xi_1$ on $C_{\delta}^{-1}(\widehat{f}(\widehat{x_i}))E_1^{u}(\widehat{f}(\widehat{x_i})))$ contains $B(0, (\|A^{-1}\|^{-1} - \gamma(1+ \gamma_0)) \alpha - \gamma \beta)$ and

$$(\|A^{-1}\|^{-1} - \gamma(1+ \gamma_0)) \alpha - \gamma \beta \geq (\e^{\chi_s - \delta} -2 \times \frac{5}{\alpha_0} \e^{-3 \delta n})\e^{-4 \delta n} - \frac{5}{\alpha_0} \e^{-3 \delta n} \e^{-8 \delta n} \geq \e^{-4 \delta n}.$$

Moreover,

\begin{equation*}
\begin{split}
\| \xi_1(0) \| &\leq (1+ \gamma_0)( \| B\| \beta + \gamma \beta + \| D^2 g_{\widehat{x_i}} \|_{B(0, R_0)} \beta^2) \\
& \leq \beta (1 + \gamma_0) ( \e^{\delta} + \frac{5}{\alpha_0} \e^{-3 \delta n} + \| D^2 g_{\widehat{x_i}} \|_{B(0, R_0)} \beta).\\
\end{split}
\end{equation*}

Now, we saw that $\| D^2 g_{\widehat{x_i}} \|_{B(0, R_0)} \leq \frac{1}{R_0}=\frac{\e^{4 \delta n}}{5}$ and $1+ \gamma_0 \leq \e^{\delta/2}$ if we take $\gamma_0$ small enough with respect to $\delta$. Then, we obtain $\| \xi_1(0) \| \leq \beta \e^{2 \delta}$.

\end{proof}

From the graph $(X', \xi_1(X'))$, we keep only the part over $B_s(0, \e^{-4 \delta n})$. Notice, that this cut-off is the same than the previous one. Namely, if we take $X' \in B_s(0, \e^{-4 \delta n})$, we have

$$\| \xi_1(X') \| \leq \| \xi_1(X')- \xi_1(0) \| + \| \xi_1(0) \| \leq \gamma_0 \| X' \| + \e^{-8 \delta n + 2 \delta} \leq e^{-4 \delta n}.$$

The projection of $(X', \xi_1(X'))$ on $C_{\delta}^{-1}(\widehat{f}(\widehat{x_i}))E^{u}(\widehat{f}(\widehat{x_i})))$ is well in $B_{s+l_1}(0, \e^{-4 \delta n})$.

We do again, what we have done, with $g_{\widehat{f}(\widehat{x_i})}$ instead of $g_{\widehat{x_i}}$, and so on.

At the end, we obtain a graph $(X', \xi_n(X'))$ over $B_s(0, \e^{-4 \delta n}) \subset C_{\delta}^{-1}(\widehat{f}^n(\widehat{x_i}))E_1^{u}(\widehat{f}^n(\widehat{x_i})))$ with $Lip \xi_n \leq \gamma_0$ and $\| \xi_n(0) \| \leq \e^{-8 \delta n + 2 \delta n}= \e^{-6 \delta n}$.

By considering all $(a_{s+1}, \cdots , a_{s+l_1}) \in B_{l_1}(0, \e^{-8 \delta n})$, we have a foliation of $W^0_n(f^n(x_i))$.

Now, in the coordinate system

$$C_{\delta}^{-1}(\widehat{f}^n(\widehat{x_{i}}))E_1^{u}(\widehat{f}^n(\widehat{x_{i}})) \oplus C_{\delta}^{-1}(\widehat{f}^n(\widehat{x_{i}}))E_2^{u}(\widehat{f}^n(\widehat{x_{i}})) \oplus C_{\delta}^{-1}(\widehat{f}^n(\widehat{x_{i}})) E^{s}(\widehat{f}^n(\widehat{x_{i}}))$$

we consider the complex planes of dimension $k-s$, given by the equations $x_1=b_1 , \cdots , x_s= b_s$ with $(b_1, \cdots , b_s) \in B_s(0 , \e^{-4 \delta n})$. The intersections $I_0=I_0(b_1, \cdots , b_s)$ of these planes with 
$W^0_n(f^n(x_i))$ have complex dimension $l_1$. We prove below that the $2 l_1$-volume of $I_0$ is bound by below by $\e^{2 \chi_{s+1} n + \cdots + 2 \chi_{s+l_1} n - 20 \delta n k}$ and with the coarea formula (see \cite{Fe} p.258), we have then

$$vol_{2(s+l_1)}(W^0_n(f^n(x_i))) \geq \int_{B_s(0 , \e^{-4 \delta n})} \int_{\pi_1^{-1}(z) \cap W^0_n(f^n(x_i))} d \mathcal{H}^{2l_1} d \mathcal{H}^{2s}(z)$$

where $\pi_1$ is the orthogonal projection on $C_{\delta}^{-1}(\widehat{f}^n(\widehat{x_{i}}))E_1^{u}(\widehat{f}^n(\widehat{x_{i}}))$, and so

\begin{equation*}
\begin{split}
vol_{2(s+l_1)}(W^0_n(f^n(x_i))) & \geq \e^{-4 \delta n \times 2s} \e^{2 \chi_{s+1} n + \cdots + 2 \chi_{s+l_1} n - 20 \delta n k} \\
& \geq \e^{2 \chi_{s+1} n + \cdots + 2 \chi_{s+l_1} n - 28 \delta n k}\\
\end{split}
\end{equation*}

which is the lower bound that we wanted.

We have to prove now the Lemma

\begin{Lem}

The $2 l_1$-volume of $I_0$ is bound by below by $\e^{2 \chi_{s+1} n + \cdots + 2 \chi_{s+l_1} n - 20 \delta n k}$.

\end{Lem}

\begin{proof}

We have

$$vol_{2l_1} g_{\widehat{f}^{n-1}(\widehat{x_{i}})}^{-1}(I_0) = \int_{I_0} \| \Lambda^{l_1} D g_{\widehat{f}^{n-1}(\widehat{x_{i}})}^{-1}(z) \|^2 d \mathcal{H}^{2 l_1}(z)$$

by using the coarea formule. Here, we consider $g_{\widehat{f}^{n-1}(\widehat{x_{i}})}^{-1}$ as an application on $I_0$ and $D g_{\widehat{f}^{n-1}(\widehat{x_{i}})}^{-1}(z)$ is the complex differential.

We need a bound by above of $\| \Lambda^{l_1} D g_{\widehat{f}^{n-1}(\widehat{x_{i}})}^{-1}(z) \|$.

First, this quantity is smaller than $\| \Lambda^{l_1} D g_{\widehat{f}^{n-1}(\widehat{x_{i}})}^{-1}(z) \|$ where $g_{\widehat{f}^{n-1}(\widehat{x_{i}})}^{-1}$ is considered on all $W^0_n(f^n(x_i))$. We have

$$\| \Lambda^{l_1} D g_{\widehat{f}^{n-1}(\widehat{x_{i}})}^{-1}(z) \| = \| D g_{\widehat{f}^{n-1}(\widehat{x_{i}})}^{-1}(z)v_1 \wedge \cdots \wedge  D g_{\widehat{f}^{n-1}(\widehat{x_{i}})}^{-1}(z)v_{l_1} \|$$

for some $v_1, \cdots , v_{l_1}$ tangent to $W^0_n(f^n(x_i))$.

Let $u_1, \cdots , u_{l_1}$ be the projections of $v_1, \cdots , v_{l_1}$ on $C_{\delta}^{-1}(\widehat{f}^n(\widehat{x_{i}}))E^{u}(\widehat{f}^n(\widehat{x_{i}}))$. For $j=1, \cdots ,l_1$, we can write $v_j=( \alpha_j, D \Phi_n(p) \alpha_j)$ where $p$ is the orthogonal projection of $z$ on $C_{\delta}^{-1}(\widehat{f}^n(\widehat{x_{i}}))E^{u}(\widehat{f}^n(\widehat{x_{i}}))$ and $u_j=( \alpha_j , 0)$.

We have $\| v_1 \wedge \cdots \wedge v_{l_1} - u_1 \wedge \cdots \wedge u_{l_1} \| = \| ( \Lambda^{l_1} G - \Lambda^{l_1} \mathcal{I})(u_1 \wedge \cdots \wedge u_{l_1}) \|$ where 

$$G(X,Y)=(X, D \Phi_n(p)(X)) =
\left(
\begin{array}{cc}
I & 0\\
D \Phi_n(p) & 0
\end{array}
\right)
\left(
\begin{array}{c}
X\\
Y
\end{array}
\right)
$$

and $\mathcal{I}= 
\left(
\begin{array}{cc}
I & 0\\
0 & 0
\end{array}
\right)$

($I$ is the matrix identity of $\Cc^{s+ l_1}$).

But, we have $\| D \Phi_n(p) \alpha_j \| \leq \gamma_0 \| \alpha_j\|$ (because $Lip \Phi_n \leq \gamma_0$) so $\| v_1 \wedge \cdots \wedge v_{l_1} - u_1 \wedge \cdots \wedge u_{l_1} \| $ is as small as we want, if we take $\gamma_0$ small.

Now,

\begin{equation*}
\begin{split}
\| \Lambda^{l_1} D g_{\widehat{f}^{n-1}(\widehat{x_{i}})}^{-1}(z) \| & = \| D g_{\widehat{f}^{n-1}(\widehat{x_{i}})}^{-1}(z)v_1 \wedge \cdots \wedge  D g_{\widehat{f}^{n-1}(\widehat{x_{i}})}^{-1}(z)v_{l_1} \| \\
& \leq \| \Lambda^{l_1} D g_{\widehat{f}^{n-1}(\widehat{x_{i}})}^{-1}(0)( u_1 \wedge \cdots \wedge u_{l_1}) \| + A+ B\\
\end{split}
\end{equation*}

with

$$A= \| ( \Lambda^{l_1} D g_{\widehat{f}^{n-1}(\widehat{x_{i}})}^{-1}(z) - \Lambda^{l_1} D g_{\widehat{f}^{n-1}(\widehat{x_{i}})}^{-1}(0))(v_1 \wedge \cdots \wedge v_{l_1} )\|$$

$$B= \| \Lambda^{l_1} D g_{\widehat{f}^{n-1}(\widehat{x_{i}})}^{-1}(0)(v_1 \wedge \cdots \wedge v_{l_1} - u_1 \wedge \cdots \wedge u_{l_1}) \|.$$

$A$ is as small as we want if we take $n$ high enough because by using the Proposition \ref{prop1}, we have $\| D^2 g_{\widehat{f}^{n-1}(\widehat{x_{i}})}^{-1}\| \leq \frac{\e^{\delta}}{\alpha_0}$ and $z \in B_{k}(0, (1+\gamma_0)\e^{-4 \delta n})$.

$B$ is smaller than $\epsilon(\gamma_0)$ with $\displaystyle \lim_{\gamma_0 \to 0} \epsilon(\gamma_0)=0$, so for $n$ high enough

\begin{equation*}
\begin{split}
\| \Lambda^{l_1} D g_{\widehat{f}^{n-1}(\widehat{x_{i}})}^{-1}(z) \| &\leq \| \Lambda^{l_1} D g_{\widehat{f}^{n-1}(\widehat{x_{i}})}^{-1}(0)_{|C_{\delta}^{-1}(\widehat{f}^n(\widehat{x_{i}}))E^{u}(\widehat{f}^n(\widehat{x_{i}}))}   \| + 2 \epsilon(\gamma_0)\\
& \leq \e^{- \chi_{s+1} - \cdots - \chi_{s+l_1} + \delta l_1} + 2 \epsilon(\gamma_0)\\
& \leq \e^{- \chi_{s+1} - \cdots - \chi_{s+l_1} + \delta l_1}(1+2\epsilon(\gamma_0)).
\end{split}
\end{equation*}

The $2l_1$-volume of $g_{\widehat{f}^{n-1}(\widehat{x_{i}})}^{-1}(I_0)$ is then bounded above by $\e^{- 2\chi_{s+1} - \cdots - 2\chi_{s+l_1} + 2 \delta l_1}(1+2 \epsilon(\gamma_0))^2 \times vol_{2l_1}(I_0)$.

Now, we take the image by $g_{\widehat{f}^{n-2}(\widehat{x_{i}})}^{-1}$ and we do the same thing and so on. At the end, we obtain that

$$vol_{2l_1} (g_{\widehat{x_{i}}}^{-1} \circ \cdots \circ  g_{\widehat{f}^{n-1}(\widehat{x_{i}})}^{-1}(I_0)) \leq \e^{- 2\chi_{s+1}n - \cdots - 2\chi_{s+l_1}n + 2n\delta l_1}(1+2 \epsilon(\gamma_0))^{2n} \times vol_{2l_1}(I_0).$$

But the intersection between $g_{\widehat{x_{i}}}^{-1} \circ \cdots \circ  g_{\widehat{f}^{n-1}(\widehat{x_{i}})}^{-1}(I_0)$ and $B_s(0, \e^{-4 \delta n}) \times \{a_{s+1}\} \times \cdots \times \{ a_{s+l_1} \} \times \{0\}^{k-s-l_1}$ is non empty for $(a_{s+1}, \cdots , a_{s+l_1}) \in B_{l_1}(0, \e^{- 8 \delta n})$, so its volume is higher than $\e^{- 8 \delta n \times 2 l_1}$. Then,

\begin{equation*}
\begin{split}
vol_{2l_1}(I_0) & \geq \e^{ 2\chi_{s+1}n + \cdots + 2\chi_{s+l_1}n - 2n \delta l_1}   \frac{\e^{- 8 \delta n \times 2 l_1}}{(1+2 \epsilon(\gamma_0))^{2n}} \\
& \geq \e^{ 2\chi_{s+1}n + \cdots + 2\chi_{s+l_1}n - 20 \delta kn}
\end{split}
\end{equation*}

for $\gamma_0$ small enough, so that $1+2 \epsilon(\gamma_0) \leq \e^{\delta}$.

This is the bound by below that we wanted: the $2(s+l_1)$-volume of the $W_n(f^n(x_i))$ are higher than $\e^{ 2\chi_{s+1}n + \cdots + 2\chi_{s+l_1}n - 30 \delta kn}$ (for $i=1, \cdots , N''$).

\end{proof}

The $2(s+l_1)$-volume of the projection $P_n(f^n(x_i))$ of $W_n(f^n(x_i))$ on $C_{\delta}^{-1}(\widehat{f}^n(\widehat{x_{i_0}}))E^{u}(\widehat{f}^n(\widehat{x_{i_0}}))$ is larger than $\e^{ 2\chi_{s+1}n + \cdots + 2\chi_{s+l_1}n - 31 \delta kn}$ because $W_n(f^n(x_i))$ is a graph $(X, \Psi_n(X))$ with $Lip \Psi_n \leq 2 \gamma_0$. The union of the $P_n(f^n(x_i))$ (for $i=1, \cdots , N''$) has a $2(s+l_1)$-volume (counted with multiplicity) bigger than

$$N'' \times \e^{ 2\chi_{s+1}n + \cdots + 2\chi_{s+l_1}n - 31 \delta kn} \geq \frac{1}{5N(\eta)} \e^{ h_{\mu}(f)n - \epsilon n} \e^{-16k \delta n} \e^{ 2\chi_{s+1}n + \cdots + 2\chi_{s+l_1}n - 31 \delta kn},$$

so we can find a point $p$ which is at least in

$$\frac{1}{5N(\eta)} \e^{ h_{\mu}(f)n - \epsilon n} \e^{-16k \delta n} \e^{ 2\chi_{s+1}n + \cdots + 2\chi_{s+l_1}n - 31 \delta kn} \epsilon_0^{-2k}$$

$P_n(f^n(x_i))$ (because these points are in $B_{s+l_1}(0, \epsilon_0)$).

Consider $\Delta'= \pi_2^{-1}(p)$, with $\pi_2$ the orthogonal projection on $C_{\delta}^{-1}(\widehat{f}^n(\widehat{x_{i_0}}))E^{u}(\widehat{f}^n(\widehat{x_{i_0}}))$, and $\Delta= \tau_{f^n(x_{i_0})} C_{\delta}(\widehat{f}^n(\widehat{x_{i_0}}))( \Delta')$. This is the $\Delta$ that we wanted. We pull-back now $\Delta$ by $f^n$ and we construct graphs in $f^{-n}(\Delta)$.

\subsubsection{\bf Construction of the graphs in $f^{-n}(\Delta)$}

By definition, $\Delta'$ cuts 

$$N_1 \geq \frac{1}{5N(\eta)} \e^{ h_{\mu}(f)n - \epsilon n}  \e^{ 2\chi_{s+1}n + \cdots + 2\chi_{s+l_1}n - 48 \delta kn} \geq \e^{ h_{\mu}(f)n - \epsilon n}  \e^{ 2\chi_{s+1}n + \cdots + 2\chi_{s+l_1}n - 50 \delta kn}$$

set $W_n(f^n(x_i))$ (for $n$ high enough).

To simplify the notations, denote by $1, \cdots , N_1$ the concerned indices $i$.

Fix $i \in \{1, \cdots , N_1 \}$. We tranfer $\Delta'$ in the coordinate system 

$$C_{\delta}^{-1}(\widehat{f}^n(\widehat{x_{i}}))E^{u}(\widehat{f}^n(\widehat{x_{i}})) \oplus C_{\delta}^{-1}(\widehat{f}^n(\widehat{x_{i}}))E^{s}(\widehat{f}^n(\widehat{x_{i}}))$$

which means that we consider $C'(\Delta')$ with $C'= C_{\delta}^{-1}(\widehat{f}^n(\widehat{x_{i}})) \tau^{-1}_{f^n(x_{i})} \tau_{f^n(x_{i_0})} C_{\delta}(\widehat{f}^n(\widehat{x_{i_0}}))$. As we proved previously, $C'(\Delta')$ is a graph $(\psi_n(Y),Y)$ at least over $B_{k-s-l_1}(0, \e^{-4 \delta n}) \subset C_{\delta}^{-1}(\widehat{f}^n(\widehat{x_{i}}))E^{s}(\widehat{f}^n(\widehat{x_{i}}))$ with $Lip \psi_n \leq 2 \gamma_0$.

We have $\Delta' \cap W_n(f^n(x_i)) \neq \emptyset$, so take $a \in \Delta' \cap W_n(f^n(x_i))$. The point $C'(a)$ is in $C'(\Delta') \cap W^0_n(f^n(x_i))$, then we can write $C'(a)=( \alpha, \Phi_n(\alpha))$ with $\alpha \in B_{s+l_1}(0, \e^{-4 \delta n})$. The fact that $\Phi_n(0)=0$ and $Lip \Phi_n \leq \gamma_0$ gives

$$\| \Phi_n(\alpha)\| \leq \gamma_0 \| \alpha\| \leq \gamma_0 \e^{-4 \delta n},$$

so the projection of $C'(a)$ on $C_{\delta}^{-1}(\widehat{f}^n(\widehat{x_{i}}))E^{s}(\widehat{f}^n(\widehat{x_{i}}))$ is inside $B_{k-s-l_1}(0, \e^{-4 \delta n})$. In particular, we can write $C'(a)= (\psi_n(\beta), \beta)$ with $\beta \in B_{k-s-l_1}(0, \e^{-4 \delta n})$ and so

\begin{equation*}
\begin{split}
\| \psi_n(0) \| &\leq \| \psi_n(0) - \psi_n(\beta) \|+ \| \psi_n(\beta) \| \leq 2 \gamma_0 \times \| \beta \| + \| \alpha \| \\
& \leq (1+ 2 \gamma_0) \e^{-4 \delta n} \leq 2 \e^{-4 \delta n}.
\end{split}
\end{equation*}

Now, we take the pull-back of this graph $(\psi_n(Y),Y)$. We need the following Lemma

\begin{Lem}{\label{lemme3}}

Fix $l=0, \cdots , n-1$. In the coordinate system

$$C_{\delta}^{-1}(\widehat{f}^{l+1}(\widehat{x_{i}}))E^{u}(\widehat{f}^{l+1}(\widehat{x_{i}})) \oplus C_{\delta}^{-1}(\widehat{f}^{l+1}(\widehat{x_{i}}))E^{s}(\widehat{f}^{l+1}(\widehat{x_{i}})),$$

we consider the graph $(\psi_{l+1}(Y),Y)$ over $B_{k-s-l_1}(0, \e^{-4 \delta n})$ with $Lip \psi_{l+1} \leq 2 \gamma_0$ and $\| \psi_{l+1} (0) \| \leq 2 \e^{-4 \delta n}$. Then the image of this graph by $g_{\widehat{f}^{l}(\widehat{x_{i}})}^{-1}$ is a graph $(\psi_{l}(Y),Y)$ at least over $B_{k-s-l_1}(0, \e^{-4 \delta n})$ with $Lip \psi_{l} \leq 2 \gamma_0$.

\end{Lem}

\begin{Proof}

We need to verify the hypothesis of the graph transform Theorem. 

In the previous coordinate system, we can write

$$g_{\widehat{f}^{l}(\widehat{x_{i}})}^{-1}(X,Y)=(AX+R(X,Y), BY+ U(X,Y))$$ 

where $(A,B)=diag((A_{\delta}^1(\widehat{f}^l(\widehat{x_i})))^{-1}, \cdots , (A_{\delta}^{p_0}(\widehat{f}^l(\widehat{x_i})))^{-1})$, $\|A\| \leq \e^{- \chi_{s+l_1} + \delta}$ and $\|B^{-1}\|^{-1} \geq \e^{- \chi_{s+l_1+1} - \delta}$.

By using the Proposition \ref{prop2} we have

\begin{equation*}
\begin{split}
\max( \| DR(X,Y)\|, \|DU(X,Y)\|) &\leq \frac{1}{r_2(\widehat{f}^l(\widehat{x_i}))} \times \|(X,Y)\| \\
&\leq \frac{1}{\alpha_0 \e^{ -\delta n}} 3 \e^{-4 \delta n} = \frac{3}{\alpha_0} \e^{-3 \delta n}\\
\end{split}
\end{equation*}

because $r_2$ is a tempered function and with $\|(X,Y)\| \leq R_0= 3 \e^{-4 \delta n}$.

Now, we verify the hypothesis of the graph transform Theorem. We have

$$\gamma \|B^{-1} \| (1+ 2 \gamma_0) \leq \frac{3}{\alpha_0} \e^{-3 \delta n} \e^{ \chi_{s+l_1+1} + \delta}(1+ 2 \gamma_0) < 1$$

$$ \frac{\|A\|2 \gamma_0 + \gamma(1+ 2 \gamma_0)}{ \|B^{-1}\|^{-1} - \gamma(1+ 2 \gamma_0)} \leq \frac{\e^{- \chi_{s+l_1} + \delta} 2\gamma_0 + \frac{3}{\alpha_0} \e^{- 3\delta n} \times 2}{\e^{- \chi_{s+l_1+1} - \delta}-\frac{3}{\alpha_0} \e^{- 3 \delta n} \times 2} \leq 2 \gamma_0$$

for $n$ high enough (take $\delta$ small to have $\e^{ - \chi_{s+l_1} +\chi_{s+l_1+1}  + 2 \delta} < 1$).

So we obtain a graph at least over $B_{k-s-l_1}(0,t)$ with

\begin{equation*}
\begin{split}
t &\geq (\|B^{-1}\|^{-1}- \gamma(1+ 2 \gamma_0))\e^{-4 \delta n} - \gamma  2 \e^{-4 \delta n} \\
&\geq (\e^{- \chi_{s+l_1+1} - \delta} - \frac{3}{\alpha_0} \e^{- 3 \delta n} \times 2) \e^{-4 \delta n} -\frac{3}{\alpha_0} \e^{-3 \delta n} 2 \e^{-4 \delta n} \\
&=  \e^{-4 \delta n} (\e^{- \chi_{s+l_1+1} - \delta} - \frac{12}{\alpha_0} \e^{- 3 \delta n}) \geq \e^{-4 \delta n}.
\end{split}
\end{equation*}

By using the graph transform Theorem, the Lemma is proved.

\end{Proof}

We apply this Lemma with $l=n-1$ to the graph $(\psi_n(Y),Y)$ and then we obtain a graph $(\psi_{n-1}(Y),Y)$ at least over $B_{k-s-l_1}(0, \e^{-4 \delta n})$ and with $Lip \psi_{l} \leq 2 \gamma_0$. In order to apply again the Lemma with $l=n-2$, we have to estimate $\psi_{n-1}(0)$.

The graph $(\psi_{n-1}(Y),Y)$ contains $g_{\widehat{f}^{n-1}(\widehat{x_{i}})}^{-1}(C'(a))$ (because $C'(a) \in W^0_n(f^n(x_i))$, so all its preimages $g_{\widehat{f}^{l}(\widehat{x_{i}})}^{-1} \circ \cdots \circ g_{\widehat{f}^{n-1}(\widehat{x_{i}})}^{-1}(C'(a))$, for $l=0, \cdots, n-1$, are in the graphs $(X, \Phi_l(X))$ used to construct $W^0_n(f^n(x_i))$). By using the same argument than for $\psi_n$, we have thus $\| \psi_{n-1}(0) \| \leq 2 \e^{-4 \delta n}$. Moreover,

$$g_{\widehat{f}^{n-1}(\widehat{x_{i}})}^{-1}(C'(a))=( \psi_{n-1}(\beta), \beta)=(\alpha, \Phi_{n-1}(\alpha))$$

with $\| \alpha \| \leq \e^{-4 \delta n}$. So,

$$\| \beta \| = \| \Phi_{n-1}(\alpha) \| = \| \Phi_{n-1}(\alpha) - \Phi_{n-1}(0) \| \leq \gamma_0 \| \alpha \| \leq \gamma_0 \e^{-4 \delta n}$$

and then this point will not be removed when we will do a cut-off over $B_{k-s-l_1}(0, \e^{-4 \delta n})$.

Now, we do this cut-off and we take the preimage by $g_{\widehat{f}^{n-2}(\widehat{x_{i}})}$ and so on. At the end, we obtain a graph $(\psi_0(Y),Y)$ over $B_{k-s-l_1}(0, \e^{-4 \delta n}) \subset C_{\delta}^{-1}(\widehat{x_{i}})E^{s}(\widehat{x_{i}})$ with $Lip \psi_0 \leq 2 \gamma_0$.

Denote by $V_i$ the image by $\tau_{x_i} \circ C_{\delta}(\widehat{x_{i}})$ of the graph $(\psi_0(Y),Y)$ (for $i=1, \cdots , N''$). First, remark that

$$V_i \subset \tau_{x_i} \circ C_{\delta}(\widehat{x_{i}}) \circ g_{\widehat{x_{i}}}^{-1} \circ \cdots \circ g_{\widehat{f}^{n-1}(\widehat{x_{i}})}^{-1}(C'(\Delta')) = f^{-n}(\tau_{f^n(x_{i_0})} \circ C_{\delta}(\widehat{f}^n(\widehat{x_{i_0}})))(\Delta')=f^{-n}(\Delta).$$

The $V_i$ are then in $f^{-n}(\Delta)$ and they are the graphs that we wanted. We use now these sets to finish the proof of the Theorem \ref{th2}.

\subsubsection{\bf End of the proof}

Firstly, the $V_i$ are $(n, 5 \delta)$-separated. Indeed, take $i \neq j$. As $x_i$ and $x_j$ are $(n, 6 \delta)$-separated, there exits $l \in \{ 0 , \cdots , n-1 \}$ such that $dist(f^l(x_i), f^l(x_j)) \geq 6 \delta$. Moreover

\begin{equation*}
\begin{split}
f^l(V_i)&= \tau_{f^l(x_i)} \circ C_{\delta}( \widehat{f}^l(\widehat{x_{i}})) \circ g_{\widehat{f}^{l-1}(\widehat{x_{i}})} \circ \cdots \circ g_{\widehat{x_{i}}} \circ C_{\delta}(\widehat{x_{i}})^{-1} \tau_{x_i}^{-1} (V_i)\\
&\subset  \tau_{f^l(x_i)} \circ C_{\delta}(  \widehat{f}^l(\widehat{x_{i}}))( \mbox{graph of  } \psi_l)
\end{split}
\end{equation*}

because $g_{\widehat{f}^{p}(\widehat{x_{i}})}( \mbox{graph of  } \psi_p) \subset  \mbox{graph of  } \psi_{p+1}$ for $p=0, \cdots , n-1$.

We have 

$$ \| \psi_l(Y) \| \leq \| \psi_l(Y)- \psi_l(0) \| + \| \psi_l(0) \| \leq 2 \gamma_0 \e^{-4 \delta n} + 2 \e^{-4 \delta n}$$

for $\|Y\| \leq \e^{-4 \delta n}$, so if $y \in  \tau_{f^l(x_i)} \circ C_{\delta}(  \widehat{f}^l(\widehat{x_{i}}))( \mbox{graph of  } \psi_l)$, we obtain 

$$dist(y,f^l(x_i)) \leq C \frac{\e^{\delta l}}{\alpha_0} (2 \gamma_0 \e^{-4 \delta n} + 2 \e^{-4 \delta n} + \e^{-4 \delta n}) \leq \frac{\delta}{2}$$

for $n$ high enough (recall that $\tau_{f^l(x_i)} \circ C_{\delta}(  \widehat{f}^l(\widehat{x_{i}}))(0)=f^l(x_i)$).

So, if $y_i \in f^l(V_i)$ and $y_j \in f^l(V_j)$ are such that $dist(f^l(V_i),f^l( V_j))= dist(y_i,y_j)$, we have

$$dist(f^l(V_i),f^l( V_j)) \geq -dist(y_i, f^l(x_i)) - dist(y_j, f^l(x_j)) + dist(f^l(x_i),f^l(x_j)) \geq 5 \delta.$$

Thus the $V_i$ are $(n, 5 \delta)$-separated.

It remains to prove that the $V_i$ are in $X_{k-s-l_1}^{5 \delta , n}$.

We have $V_i= \tau_{x_i} \circ C_{\delta}(\widehat{x_{i}}) ( \mbox{graph of  } \psi_0)$ and $(\psi_0(Y),Y)$ is a graph over $B_{k-s-l_1}(0, \e^{-4 \delta n})$ with $\| \psi_0(0) \| \leq 2 \e^{-4 \delta n}$ and $Lip(\psi_0) \leq 2 \gamma_0$.

The image by $C_{\delta}(\widehat{x_{i}})$ is a graph $(\zeta_0(Y),Y)$, in the coordinate system $E^{u}(\widehat{x_{i}}) \oplus E^{s}(\widehat{x_{i}})$, at least over $B_{k-s-l_1}(0, \alpha_0 \e^{-4 \delta n})$ with $Lip(\zeta_0) \leq \frac{2 \gamma_0}{ \alpha_0^2} \leq 1.$
 
Now $\tau_{x_i}= \psi \circ t$ where $t$ is a translation of vector $(t_1,t_2) \in \Cc^{s+l_1} \oplus \Cc^{k-s-l_1}$ and $\psi$ is a chart of $X$, so

$$t(\zeta_0(Y),Y)=(\zeta_0(Y)+t_1,Y+t_2)=(\theta_0(Y'),Y')$$

with $Y'=Y+t_2$ and $\theta_0(Y')=\zeta_0(Y)+t_1= \zeta_0(Y' - t_2)+t_1$

is a graph $(\theta_0(Y'),Y')$ over at least $B_{k-s-l_1}(t_2, \e^{-5 \delta n})$ with $Lip( \theta_0)= Lip( \zeta_0) \leq 1$. This is an element in $X_{k-s-l_1}^{5 \delta , n}$ (by construction the iterates do not meet too the indeterminacy set $I$).

Up to divide $\Delta$ into $C(\delta)$ pieces (and so by changing $N_1$ into $\frac{N_1}{C(\delta)}$), we can consider that $\Delta$ is in $X_{k-s-l_1}^{5 \delta}$.

So, we have

\begin{equation*}
\begin{split}
&\frac{1}{n} \log \sup_{\Delta \in X_{k-s-l_1}^{5 \delta} } ( \max \# E \mbox{ , } E \mbox{   } (n,5 \delta) \mbox{-separated   } E \subset X_{k-s-l_1}^{5 \delta , n} \mbox{   and   } \forall W \in E \mbox{         } W \subset f^{-n}(\Delta) )  \\
& \geq \frac{1}{n} \log \frac{\e^{ h_{\mu}(f)n - \epsilon n + 2\chi_{s+1}n + \cdots + 2\chi_{s+l_1}n - 50 \delta kn}}{C(\delta)}\\
& =  h_{\mu}(f) - \epsilon + 2\chi_{s+1} + \cdots + 2\chi_{s+l_1} - 50 \delta k - \frac{1}{n} \log C(\delta)
\end{split}
\end{equation*}

and then

\begin{equation*}
\begin{split}
 &\overline{\lim_n} \frac{1}{n} \log \sup_{\Delta \in X_{k-s-l_1}^{5 \delta} } ( \max \# E \mbox{ , }
E \mbox{   } (n,5 \delta) \mbox{-separated   } E \subset X_{k-s-l_1}^{5 \delta , n} \mbox{   and   } \forall W \in E \mbox{        } W \subset f^{-n}(\Delta) )  \\
& \geq  h_{\mu}(f) - \epsilon + 2\chi_{s+1} + \cdots + 2\chi_{s+l_1} - 50 \delta k
\end{split}
\end{equation*}

and
 
\begin{equation*}
\begin{split}
h_{(k-s-l_1,k-s-l_1)}^{top}(f)= \overline{\lim_{\delta \to 0}} \overline{\lim_n} \frac{1}{n} \log \sup_{\Delta \in X_{k-s-l_1}^{\delta} } ( & \max \# E \mbox{ , } E \mbox{   } (n,\delta) \mbox{-separated   } E \subset X_{k-s-l_1}^{\delta , n}\\ & \mbox{   and   } \forall W \in E \mbox{   we have   } W \subset f^{-n}(\Delta) )  \\
\geq h_{\mu}(f) - \epsilon + 2\chi_{s+1} + \cdots + 2\chi_{s+l_1},
\end{split}
\end{equation*}

and the result follows by letting $\epsilon \to 0$.

\subsection{\bf{Proof of the Theorem \ref{th1}}}

The beginning is almost the same than in the proof of the Theorem \ref{th2}.

We construct $x_1, \cdots , x_N \in \Lambda$ which are $(n, 9 \delta)$-separated (instead of $(n, 6\delta)$-separated in Theorem \ref{th2}) and $N \geq \frac{\e^{h_{\mu}(f)n - \epsilon n}}{5}$. We can find $N'= \frac{\e^{h_{\mu}(f)n - \epsilon n}}{5} \e^{-16 \delta k n}$ indices $i$ with $f^n(x_i)$ in the same ball $B(x, \e^{-8 \delta  n})$.

Consider the notations just before the Lemma \ref{lem2}.

We pushforward $n$ times a graph $(X', \xi_0(X'))$ over a part of $C_{\delta}^{-1}(\widehat{x_i})E_1^{u}(\widehat{x_i})$ (with $\xi_0 \equiv 0$).

As previously, we obtain a graph $(X', \xi_n(X'))$ over $B_s(0, \e^{-4 \delta n}) \subset C_{\delta}^{-1}(\widehat{f}^n(\widehat{x_i}))E_1^{u}(\widehat{f}^n(\widehat{x_i}))$ with $Lip(\xi_n) \leq \gamma_0$ and $\xi_n(0)=0$.

We want to considerate these graphs in the coordinate system 

$$C_{\delta}^{-1}(\widehat{f}^n(\widehat{x_{i_0}}))E_1^{u}(\widehat{f}^n(\widehat{x_{i_0}})) \oplus C_{\delta}^{-1}(\widehat{f}^n(\widehat{x_{i_0}}))E_2^{u}(\widehat{f}^n(\widehat{x_{i_0}})) \oplus C_{\delta}^{-1}(\widehat{f}^n(\widehat{x_{i_0}})) E^{s}(\widehat{f}^n(\widehat{x_{i_0}})),$$

it means to take the image by 

$$C=C_{\delta}^{-1}(\widehat{f}^n(\widehat{x_{i_0}})) \tau_{f^n(x_{i_0})}^{-1} \tau_{f^n(x_{i})} C_{\delta}(\widehat{f}^n(\widehat{x_{i}})).$$

As done previously, we obtain a graph $(X', \zeta_n(X'))$ over a part of $C_{\delta}^{-1}(\widehat{f}^n(\widehat{x_{i_0}}))E_1^{u}(\widehat{f}^n(\widehat{x_{i_0}}))$ with $Lip(\zeta_n) \leq 2 \gamma_0$ (because if we divide again $N'$ by $N(\eta)$, we can suppose that $d(\widehat{f}^n(\widehat{x_i}), \widehat{f}^n(\widehat{x_{i_0}}) < \eta$).

The $2s$-volume of the graph $(X', \xi_n(X'))$ is higher than $\e^{-8k \delta n}$, so the $2s$-volume of the graph $(X', \zeta_n(X'))$ is bounded by below by $\epsilon(\gamma_0) \e^{-8k \delta n}$. 

Denote by $W_n(f^n(x_i))$ these graphs. The $2s$-volume of the projection of the union of these graphs on $C_{\delta}^{-1}(\widehat{f}^n(\widehat{x_{i_0}}))E_1^{u}(\widehat{f}^n(\widehat{x_{i_0}}))$, counted with multiplicity, is higher than

$$\epsilon(\gamma_0) \e^{-8k \delta n} \frac{\e^{h_{\mu}(f)n - \epsilon n}}{5N(\eta)} \e^{-16 \delta k n},$$

so we can find a point $P \in C_{\delta}^{-1}(\widehat{f}^n(\widehat{x_{i_0}}))E_1^{u}(\widehat{f}^n(\widehat{x_{i_0}}))$ covered at least

$$\epsilon_0^{-2k} \epsilon(\gamma_0) \e^{-8k \delta n} \frac{\e^{h_{\mu}(f)n - \epsilon n}}{5N(\eta)} \e^{-16 \delta k n}$$

times.

Consider $\Delta' = \pi_{1,1}^{-1}(P)$ where $\pi_{1,1}$ is the orthogonal projection on $C_{\delta}^{-1}(\widehat{f}^n(\widehat{x_{i_0}}))E_1^{u}(\widehat{f}^n(\widehat{x_{i_0}}))$.

It is $\Delta= \tau_{f^n(x_{i_0})} \circ C_{\delta}(  \widehat{f}^n(\widehat{x_{i_0}}))(\Delta')$ that we pull-back and for which we construct graphs inside $f^{-n}(\Delta)$.

By construction, $\Delta'$ meets $N_1 \geq \e^{h_{\mu}(f)n - \epsilon n} \e^{-25k \delta n}$ sets $W_n(f^n(x_i))$ (for $n$ high enough). Denote $i=1, \cdots , N_1$ these indices to simplify the notations.

We want to considerate $\Delta'$ in the coordinate system 

$$C_{\delta}^{-1}(\widehat{f}^n(\widehat{x_{i}}))E_1^{u}(\widehat{f}^n(\widehat{x_{i}})) \oplus C_{\delta}^{-1}(\widehat{f}^n(\widehat{x_{i}}))E_2^{u}(\widehat{f}^n(\widehat{x_{i}})) \oplus C_{\delta}^{-1}(\widehat{f}^n(\widehat{x_{i}})) E^{s}(\widehat{f}^n(\widehat{x_{i}}))$$.

It means to take the image by 

$$C'=C_{\delta}^{-1}(\widehat{f}^n(\widehat{x_{i}})) \tau_{f^n(x_{i})}^{-1} \tau_{f^n(x_{i_0})} C_{\delta}(\widehat{f}^n(\widehat{x_{i_0}})).$$

As previously, $C'(\Delta')$ is a graph $(\Phi_n(Y),Y)$ over $B_{k-s}(0, \e^{-4\delta n}) \subset C_{\delta}^{-1}(\widehat{f}^n(\widehat{x_{i}}))E_2^{u}(\widehat{f}^n(\widehat{x_{i}})) \oplus C_{\delta}^{-1}(\widehat{f}^n(\widehat{x_{i_0}})) E^{s}(\widehat{f}^n(\widehat{x_{i_0}}))$ with $Lip(\Phi_n) \leq 2 \gamma_0$.

Take $a \in \Delta' \cap W_n(f^n(x_i))$. The point $C'(a)$ is in the graph $(X', \xi_n(X'))$, so $C'(a)=(\alpha, \xi_n(\alpha))$ with $\alpha \in B_s(0, \e^{-4\delta n})$. Now

$$\| \xi_n(\alpha) \| \leq \| \xi_n(\alpha) - \xi_n(0) \| + \| \xi_n(0) \| \leq \gamma_0 \| \alpha \| \leq \gamma_0 \e^{-4\delta n}$$

so the projection of $C'(a)$ on $C_{\delta}^{-1}(\widehat{f}^n(\widehat{x_{i}}))E_2^{u}(\widehat{f}^n(\widehat{x_{i}})) \oplus C_{\delta}^{-1}(\widehat{f}^n(\widehat{x_{i}})) E^{s}(\widehat{f}^n(\widehat{x_{i}}))$ is inside $B_{k-s}(0,\e^{-4\delta n})$.

So $C'(a)= (\Phi_n(\beta),\beta)$ with $\beta \in B_{k-s}(0,\e^{-4\delta n})$ which implies that

$$\| \Phi_n(0) \| \leq \| \Phi_n(0) - \Phi_n(\beta) \| + \| \Phi_n(\beta) \| \leq 2 \gamma_0 \e^{-4\delta n} + \| \alpha\| \leq 2 \e^{-4\delta n}.$$

We take the pull-back of the graph $(\Phi_n(Y),Y)$ by $g_{\widehat{f}^{n-1}(\widehat{x_{i}})}$ but this time we can have a Lyapounov exponent equal to $0$. To solve this problem, we use this Lemma:

\begin{Lem}

Fix $l=0, \cdots , n-1$. In the coordinate system

$$C_{\delta}^{-1}(\widehat{f}^{l+1}(\widehat{x_{i}}))E_1^{u}(\widehat{f}^{l+1}(\widehat{x_{i}})) \oplus C_{\delta}^{-1}(\widehat{f}^{l+1}(\widehat{x_{i}}))E_2^{u}(\widehat{f}^{l+1}(\widehat{x_{i}})) \oplus C_{\delta}^{-1}(\widehat{f}^{l+1}(\widehat{x_{i}}))E^{s}(\widehat{f}^{l+1}(\widehat{x_{i}})),$$

we consider the graph $(\Phi_{l+1}(Y),Y)$ over 

$$B_{k-s}(0, \e^{-4 \delta n - 2 \delta (n-l-1)}) \subset C_{\delta}^{-1}(\widehat{f}^{l+1}(\widehat{x_{i}}))E_2^{u}(\widehat{f}^{l+1}(\widehat{x_{i}})) \oplus C_{\delta}^{-1}(\widehat{f}^{l+1}(\widehat{x_{i}}))E^{s}(\widehat{f}^{l+1}(\widehat{x_{i}}))$$

with $Lip \Phi_{l+1} \leq 2 \gamma_0$ and $\| \Phi_{l+1} (0) \| \leq 2 \e^{-4 \delta n}$. Then the image of this graph by $g_{\widehat{f}^{l}(\widehat{x_{i}})}^{-1}$ is a graph $(\Phi_{l}(Y),Y)$ at least over $B_{k-s}(0, \e^{-4 \delta n - 2 \delta (n-l-1) -2 \delta})$ with $Lip \Phi_{l} \leq 2 \gamma_0$ and $\| \Phi_l(0) \| \leq 2 \e^{-4 \delta n}$.
 
\end{Lem}

\begin{proof}

We are in the same situation than in Lemma \ref{lemme3} with this time $\| A \| \leq \e^{- \chi_s + \delta}$ and $\|B^{-1} \|^{-1} \geq \e^{- \delta}$.

We have

$$\gamma \|B^{-1} \| (1+ 2 \gamma_0) \leq \frac{3}{\alpha_0} \e^{-3 \delta n} \e^{  \delta}(1+ 2 \gamma_0) < 1$$

$$ \frac{\|A\|2 \gamma_0 + \gamma(1+2 \gamma_0)}{ \|B^{-1}\|^{-1} - \gamma(1+2  \gamma_0)} \leq \frac{\e^{- \chi_{s} + \delta} 2\gamma_0 + \frac{3}{\alpha_0} \e^{-3 \delta n} \times 2}{\e^{ - \delta}-\frac{3}{\alpha_0} \e^{-3 \delta n} \times 2} \leq 2 \gamma_0$$

for $n$ high enough (take $\delta$ small to have $\e^{ - \chi_{s} + 2 \delta} < 1$).

So it is a graph at least over $B_{k-s}(0,t)$ with

\begin{equation*}
\begin{split}
t &\geq (\|B^{-1}\|^{-1}- \gamma(1+ 2 \gamma_0))\e^{-4 \delta n - 2\delta (n-l-1)} - \gamma  2 \e^{-4 \delta n}\\
 &\geq (\e^{ - \delta} - \frac{3}{\alpha_0} \e^{-3 \delta n} \times 2) \e^{-4 \delta n - 2\delta (n-l-1)} -\frac{3}{\alpha_0} \e^{- 3 \delta n} 2 \e^{-4 \delta n} \\
&=  \e^{-4 \delta n - 2\delta (n-l-1)} (\e^{ - \delta} - \frac{6}{\alpha_0} \e^{-3  \delta n} - \frac{6}{\alpha_0} \e^{-3  \delta n+ 2\delta (n-l-1)} ) \\
&\geq \e^{-4 \delta n - 2\delta (n-l-1)- 2 \delta}.
\end{split}
\end{equation*}

And we have

\begin{equation*}
\begin{split}
\| \Phi_l(0)\| &\leq (1+ 2 \gamma_0)2 \e^{-4 \delta n} (\|A\| + \gamma + \|D^2 g_{\widehat{f}^{l}(\widehat{x_{i}})}^{-1}\|_{B_k(0, R_0)} 2 \e^{-4 \delta n})\\
&\leq  \e^{\delta/2} 2 \e^{-4 \delta n}(\e^{- \chi_s + \delta} + \frac{3}{\alpha_0} \e^{- 3 \delta n} + \frac{1}{\alpha_0} \e^{ \delta n} 2 \e^{-4 \delta n})\\
&\leq 2 \e^{-4 \delta n}\\
\end{split}
\end{equation*}

(we can take $1+ 2\gamma_0 \leq \e^{\delta/2}$ and $\e^{- \chi_s + \delta} \leq \e^{- \delta}$ if we want).

By using the graph transform Theorem, the Lemma is proved.

\end{proof}

We take the pull-back of the graph $(\Phi_n(Y),Y)$ by $g_{\widehat{f}^{n-1}(\widehat{x_{i}})}, \cdots ,g_{\widehat{f}(\widehat{x_{i}})} $ and at the end we obtain a graph $(\Phi_0(Y),Y)$ over 

$$B_{k-s}(0,\e^{-6 \delta n}) \subset C_{\delta}^{-1}(\widehat{x_{i}})E_2^{u}(\widehat{x_{i}}) \oplus C_{\delta}^{-1}(\widehat{x_{i}})E^{s}(\widehat{x_{i}})$$ 

with $Lip(\Phi_0) \leq 2 \gamma_0$ and $\| \Phi_0(0) \| \leq 2 \e^{- 4 \delta n}$.

Now, we follow the end of the proof of the Theorem 2: we denote by $V_i$ the image by  $\tau_{x_i} \circ C_{\delta}(\widehat{x_{i}})$ of the graphs $(\Phi_0(Y),Y)$ (for $i=1, \cdots , N_1$). We have $V_i \subset f^{-n}(\Delta)$, the $V_i$ are in $X_{k-s}^{n, 7 \delta}$ and the $V_i$ are $(n, 7 \delta)$-separated. 

Up to divide $\Delta$ into $C(\delta)$ pieces (and so by changing $N_1$ into $\frac{N_1}{C(\delta)}$), we can consider that $\Delta$ is in $X_{k-s}^{7 \delta}$. Thus, we have

\begin{equation*}
\begin{split}
  &\frac{1}{n} \log \sup_{\Delta \in X_{k-s}^{7 \delta} } ( \max \# E \mbox{ , }
E \mbox{   } (n,7 \delta) \mbox{-separated   } E \subset X_{k-s}^{7 \delta , n} \mbox{   and   } \forall W \in E \mbox{          } W \subset f^{-n}(\Delta) )  \\
& \geq \frac{1}{n} \log \frac{\e^{h_{\mu}(f)n - \epsilon n} \e^{-25k \delta n}}{C(\delta)} =  h_{\mu}(f) - \epsilon -25  \delta k - \frac{1}{n} \log C(\delta).\\
\end{split}
\end{equation*}

And we conclude by taking the limits as in the end of the proof of the Theorem 2.

It remains to prove the Theorem 3.

\subsection{\bf{Proof of the Theorem \ref{th3}}}

The beginning is almost the same than in the proof of the Theorem \ref{th2}.

We construct $x_1, \cdots , x_N \in \Lambda$ which are $(n, 2 \delta)$-separated (instead of $(n, 9 \delta)$ in the Theorem \ref{th2}) and $N \geq \frac{\e^{h_{\mu}(f)n - \epsilon n}}{5}$. We can find $N'= \frac{\e^{h_{\mu}(f)n - \epsilon n}}{5} \e^{-16 \delta k n}$ indices $i$ with $f^n(x_i)$ in the same ball $B(x, \e^{-8 \delta  n})$.

We consider the coordinate system

$$C_{\delta}^{-1}(\widehat{x_{i}})E_1^{u}(\widehat{x_{i}}) \oplus C_{\delta}^{-1}(\widehat{x_{i}})E_2^{u}(\widehat{x_{i}}) \oplus C_{\delta}^{-1}(\widehat{x_{i}})E^{s}(\widehat{x_{i}})$$ 

and the set $B_{s}(0, \e^{-4 \delta n}) \times B_{k-s}(0, \e^{-8 \delta n})$ with $B_{s}(0, \e^{-4 \delta n}) \subset C_{\delta}^{-1}(\widehat{x_{i}})E_1^{u}(\widehat{x_{i}})$ and $B_{k-s}(0, \e^{-8 \delta n}) \subset C_{\delta}^{-1}(\widehat{x_{i}})E_2^{u}(\widehat{x_{i}}) \oplus C_{\delta}^{-1}(\widehat{x_{i}})E^{s}(\widehat{x_{i}})$.

We take the image of this set by $g_{\widehat{x_{i}}}$ and we do a cut-off with $B_k(0, \e^{-4 \delta n})$ and so on with $g_{\widehat{f}(\widehat{x_{i}})} , \cdots ,g_{\widehat{f}^{n-1}(\widehat{x_{i}})}$. At the end, we obtain a set $W^0_n(f^n(x_i))$ in the coordinate system

$$C_{\delta}^{-1}(\widehat{f}^{n}(\widehat{x_{i}}))E_1^{u}(\widehat{f}^{n}(\widehat{x_{i}})) \oplus C_{\delta}^{-1}(\widehat{f}^{n}(\widehat{x_{i}}))E_2^{u}(\widehat{f}^{n}(\widehat{x_{i}})) \oplus C_{\delta}^{-1}(\widehat{f}^{n}(\widehat{x_{i}}))E^{s}(\widehat{f}^{n}(\widehat{x_{i}})).$$

We want to bound by below the $2k$-volume (counted with multiplicity) of $W^0_n(f^n(x_i))$. First, we prove that this multiplicity is equal to $1$:

\begin{Lem}

For $l=0, \cdots, n-1$, $g_{\widehat{f}^{l}(\widehat{x_{i}})}$ is one-to-one on $B_k(0, \e^{-4 \delta n})$. 

\end{Lem}

\begin{proof}

We use the same method than in \cite{BD} Lemma 2.

By using the Proposition \ref{prop1} and $r_1(\widehat{f}^l(\widehat{x})) \geq \alpha_0 \e^{- \delta l} \geq \e^{-4 \delta n}$, for $z \in B_k(0, \e^{-4 \delta n})$, we have

\begin{equation*}
\begin{split}
\|Id - (Dg_{\widehat{f}^{l}(\widehat{x_{i}})}(0))^{-1} \circ Dg_{\widehat{f}^{l}(\widehat{x_{i}})}(z)\| &\leq \|(Dg_{\widehat{f}^{l}(\widehat{x_{i}})}(0))^{-1}\| \|D^2 g_{\widehat{f}^{l}(\widehat{x_{i}})}\|_{B_k(0,\e^{-4 \delta n})} \|z\| \\
& \leq \e^{-\chi_k + \delta} \frac{\e^{\delta l}}{\alpha_0} \e^{-4 \delta n} < \frac{1}{2}.\\  
\end{split}
\end{equation*}

So $Lip(Id - (Dg_{\widehat{f}^{l}(\widehat{x_{i}})}(0))^{-1} \circ g_{\widehat{f}^{l}(\widehat{x_{i}})}) < \frac{1}{2}$ and this gives that $(Dg_{\widehat{f}^{l}(\widehat{x_{i}})}(0))^{-1} \circ g_{\widehat{f}^{l}(\widehat{x_{i}})}$ is one-to-one for $z \in B_k(0, \e^{-4 \delta n})$ and thus $g_{\widehat{f}^{l}(\widehat{x_{i}})}$ too.

\end{proof}

So there is no multiplicity for the $W^0_n(f^n(x_i))$. To bound by below their volume, we use the technics of foliation as in the proof of the Theorem 2.

We consider a slice (of the initial set) $B_{s}(0, \e^{-4 \delta n}) \times \{a_{s+1}\} \times  \cdots \times \{a_k\}$ with $(a_{s+1}, \cdots , a_k) \in B_{k-s}(0, \e^{-8 \delta n})$. This slice is a graph $(X', \xi_0(X'))$ over $B_{s}(0, \e^{-4 \delta n}) \subset C_{\delta}^{-1}(\widehat{x_{i}})E_1^{u}(\widehat{x_{i}})$. We push forward $n$ times this graph, we do cut-off with $B_{s}(0, \e^{-4 \delta n})$ and we obtain (see the proof of the Theorem 2 around the Lemma \ref{lem2}) a graph $(X', \xi_n(X'))$ over $B_{s}(0, \e^{-4 \delta n})$ with $Lip(\xi_n) \leq  \gamma_0$ and $\| \xi_n(0) \| \leq \e^{-8 \delta n+ 2 \delta n}=\e^{-6 \delta n}$.

By changing $(a_{s+1}, \cdots , a_k) \in B_{k-s}(0, \e^{-8 \delta n})$, we obtain a foliation of $W^0_n(f^n(x_i))$.

Now, we consider the complex plane of dimension $k-s$ given by the equations $x_1=b_1, \cdots , x_s=b_s$ with $(b_1 , \cdots , b_s) \in B_{s}(0, \e^{-4 \delta n})$.

The intersections $I_0=I_0(b_1 , \cdots , b_s)$ of these planes with $W^0_n(f^n(x_i))$ have dimension $k-s$. We want to bound by below the $2(k-s)$-volume of $I_0$ by using the same method than in the proof of the Theorem 2.

We have

$$vol_{2(k-s)} g_{\widehat{f}^{n-1}(\widehat{x_{i}})}^{-1}(I_0) = \int_{I_0} \| \Lambda^{k-s} D g_{\widehat{f}^{n-1}(\widehat{x_{i}})}^{-1}(z) \|^2 d \mathcal{H}^{2 (k-s)}(z)$$

and here directly

\begin{equation*}
\begin{split}
\| \Lambda^{k-s} D g_{\widehat{f}^{n-1}(\widehat{x_{i}})}^{-1}(z) \| &\leq \| \Lambda^{k-s} D g_{\widehat{f}^{n-1}(\widehat{x_{i}})}^{-1}(z)- \Lambda^{k-s} D g_{\widehat{f}^{n-1}(\widehat{x_{i}})}^{-1}(0) \|+ \| \Lambda^{k-s} D g_{\widehat{f}^{n-1}(\widehat{x_{i}})}^{-1}(0)\| \\
&\leq \epsilon(\gamma_0) + \e^{- \chi_{s+1} - \cdots -  \chi_{k} + \delta k} \\
&\leq (1+ \epsilon(\gamma_0)) \e^{- \chi_{s+1} - \cdots -  \chi_{k} + \delta k} \\
\end{split}
\end{equation*}

because $\| \Lambda^{k-s} D g_{\widehat{f}^{n-1}(\widehat{x_{i}})}^{-1}(z)- \Lambda^{k-s} D g_{\widehat{f}^{n-1}(\widehat{x_{i}})}^{-1}(0) \|$ is as small as we want if we take $n$ high enough by using the Proposition \ref{prop1} (we have $\| D^2 g_{\widehat{f}^{n-1}(\widehat{x_{i}})}^{-1}\| \leq \frac{\e^{\delta}}{\alpha_0}$ and $z \in B_{k}(0, \e^{-4 \delta n})$).

Now, we take the image by $g_{\widehat{f}^{n-2}(\widehat{x_{i}})}^{-1}$ and we do the same thing and so on. At the end, we obtain that

$$vol_{2(k-s)} (g_{\widehat{x_{i}}}^{-1} \circ \cdots \circ g_{\widehat{f}^{n-1}(\widehat{x_{i}})}^{-1} (I_0)) \leq \e^{- 2\chi_{s+1}n - \cdots - 2\chi_{k}n + 2n\delta k}(1+ \epsilon(\gamma_0))^{2n} \times vol_{2(k-s)}(I_0).$$

But the intersection between $g_{\widehat{x_{i}}}^{-1} \circ \cdots \circ g_{\widehat{f}^{n-1}(\widehat{x_{i}})}^{-1} (I_0)$ and $B_s(0, \e^{-4 \delta n}) \times \{a_{s+1}\} \times \cdots \times \{ a_{k} \} $ is not empty for $(a_{s+1}, \cdots , a_{k}) \in B_{k-s}(0, \e^{- 8 \delta n})$, so its $2(k-s)$-volume is higher than $\e^{- 8 \delta n \times 2 k}$. Then,

\begin{equation*}
\begin{split}
vol_{2(k-s)}(I_0) & \geq \e^{ 2\chi_{s+1}n + \cdots + 2\chi_{k}n - 2n \delta k}   \frac{\e^{- 8 \delta n \times 2 k}}{(1+ \epsilon(\gamma_0))^{2n}} \\
& \geq \e^{ 2\chi_{s+1}n + \cdots + 2\chi_{k}n - 20 \delta kn}
\end{split}
\end{equation*}

for $\gamma_0$ small enough, so that $1+ \epsilon(\gamma_0) \leq \e^{\delta}$.

This is the bound by below that we wanted: the $2k$-volume of the $W^0_n(f^n(x_i))$ are higher than $\e^{ 2\chi_{s+1}n + \cdots + 2\chi_{k}n - 28 \delta kn}$ (for $i=1, \cdots , N'$) by using the coarea formula.

We consider $W_n(f^n(x_i))= \tau_{f^n(x_{i})} C_{\delta}(\widehat{f}^n(\widehat{x_{i}})) (W^0_n(f^n(x_i)))$ and its $2k$-volume is bound by below by $\e^{ 2\chi_{s+1}n + \cdots + 2\chi_{k}n - 30 \delta kn}$. There are $N'$ such sets and the volume of $X$ is less than $1$ (if we want), so we can find a point $y$ inside at least 

\begin{equation*}
\begin{split}
N''&=\e^{ 2\chi_{s+1}n + \cdots + 2\chi_{k}n - 30 \delta kn} \times N' \\
&= \e^{ 2\chi_{s+1}n + \cdots + 2\chi_{k}n - 30 \delta kn}  \frac{\e^{h_{\mu}(f)n - \epsilon n}}{5} \e^{-16 \delta k n}.\\
\end{split}
\end{equation*}

Remark that we can suppose $h_{\mu}(f)+ 2\chi_{s+1} + \cdots + 2\chi_{k}>0$, otherwise there is nothing to do to prove the Theorem.

Denote by $W_n(f^n(x_1)), \cdots , W_n(f^n(x_{N''}))$ the corresponding sets to symplify the notations.

So, we can find $y_1, \cdots ,y_{N''}$ with $y_i \in B_s(0, \e^{-4 \delta n}) \times B_{k-s}(0, \e^{-8 \delta n})$ (that we used to construct $W^0_n(f^n(x_i))$) and $z_i= \tau_{x_i} \circ C_{\delta}(\widehat{x_{i}})(y_i) \in f^{-n}(y)$.

To finish, we have

\begin{Lem}

The points $z_i$ are $(n, \delta)$-separated.

\end{Lem}

\begin{proof}

Fix $i \neq j$. The points $x_i$ and $x_j$ are $(n, 2 \delta)$-separated so there exists $l \in \{0, \cdots ,n-1\}$ such that $dist(f^l(x_i), f^l(x_j)) \geq 2 \delta$. 

We have done a cut-off so $g_{\widehat{f}^{l-1}(\widehat{x_{i}})} \circ \cdot \circ g_{\widehat{x_{i}}}(y_i) \in B_k(0, \e^{-4 \delta n})$ and thus

$$f^l(z_i)= \tau_{f^l(x_{i})} C_{\delta}(\widehat{f}^l(\widehat{x_{i}}))  g_{\widehat{f}^{l-1}(\widehat{x_{i}})} \circ \cdot \circ g_{\widehat{x_{i}}}(y_i) \in B(f^l(x_i), \frac{K \e^{\delta l}}{\alpha_0} \e^{-4 \delta n})$$

with $K$ which depends only of $X$, so 

$$dist(f^l(z_i),f^l(z_j)) \geq 2 \delta - 2 \frac{K \e^{\delta l}}{\alpha_0} \e^{-4 \delta n} \geq \delta$$

and it proves the Lemma.

\end{proof}

We have

\begin{equation*}
\begin{split}
\frac{1}{n} \log \sup_{y \in X } ( \max \# E \mbox{ , } &E \mbox{   } (n, \delta) \mbox{-separated   }  \mbox{   and   } E  \subset f^{-n}(y) )\\
& \geq \frac{1}{n} \log \frac{\e^{ 2\chi_{s+1}n + \cdots + 2\chi_{k}n}  \e^{h_{\mu}(f)n - \epsilon n} \e^{-46 k \delta n}}{5} \\
&=  h_{\mu}(f) + 2\chi_{s+1} + \cdots + 2\chi_{k} - \epsilon -46  \delta k - \frac{1}{n} \log 5
\end{split}
\end{equation*}

and we conclude by taking the limits as in the end of the proof of the Theorem 2.

\bigskip

\noindent Henry De Thélin, Université Paris 13, Sorbonne Paris Nord, LAGA, CNRS (UMR 7539), F-93430, Villetaneuse, France.  

\noindent Email: {\tt dethelin@math.univ-paris13.fr}

\end{document}